\pdfoutput=1 
\def\isarxiv{1}
\ifdefined\isarxiv
\documentclass[11pt]{article}
\else
\documentclass[a4paper,UKenglish, autoref, thm-restate]{lipics-v2021}
\fi
\usepackage{amsmath}
\usepackage{amsthm}
\usepackage{amssymb}
\usepackage{color}
\ifdefined\isarxiv
\usepackage[english]{babel}
\fi
\usepackage{algorithm}
\usepackage{algpseudocode}
\usepackage{scrextend}
\usepackage[T1]{fontenc}
\usepackage{bbm}
\usepackage{comment}
\ifdefined\isarxiv
\fi

\let\C\relax
\usepackage{tikz}
\ifdefined\isarxiv
\usepackage{hyperref}  %
\hypersetup{colorlinks=true,citecolor=blue,linkcolor=blue} 
\usetikzlibrary{arrows}

\usepackage[margin=1in]{geometry}
\fi
\graphicspath{{./figs/}}

\ifdefined\isarxiv
\author{
Ruizhe Zhang\thanks{\texttt{ruizhe@utexas.edu}. The University of Texas at Austin. Research supported by the University Graduate Continuing Fellowship.}
\and
Xinzhi Zhang\thanks{\texttt{xinzhi20@cs.washington.edu}. University of Washington. {Research supported by  NSF grant CCF-1813135.}} 
}
\else

\author{Ruizhe Zhang}{The University of Texas at Austin, USA}{ruizhe@utexas.edu}{}{Supported by the University Graduate Continuing Fellowship from UT
Austin.}
\author{Xinzhi Zhang}{University of Washington, USA}{xinzhi20@cs.washington.edu}{}{Research supported by  NSF grant CCF-1813135, NSF CAREER Award and Packard Fellowships.}
\authorrunning{R. Zhang and X. Zhang}
\Copyright{Ruizhe Zhang and Xinzhi Zhang}

\ccsdesc[100]{Theory of computation~Randomness, geometry and discrete structures}
\keywords{Kadison-Singer conjecture, Hyperbolic polynomials, Strongly-Rayleigh distributions, Interlacing polynomials}
\fi

\title{A Hyperbolic Extension of Kadison-Singer Type Results}%

\ifdefined\isarxiv
\makeatletter
\newtheorem*{rep@theorem}{\rep@title}
\newcommand{\newreptheorem}[2]{%
\newenvironment{rep#1}[1]{%
 \def\rep@title{#2 \ref{##1}}%
 \begin{rep@theorem}}%
 {\end{rep@theorem}}}
\makeatother

\newtheorem{theorem}{Theorem}[section]
\newreptheorem{theorem}{Theorem}
\newtheorem{lemma}[theorem]{Lemma}
\newreptheorem{lemma}{Lemma}

\newtheorem{definition}[theorem]{Definition}

\newtheorem{proposition}[theorem]{Proposition}
\newtheorem{corollary}[theorem]{Corollary}

\newtheorem{remark}[theorem]{Remark}
\newtheorem{claim}[theorem]{Claim}

\fi
\newtheorem{fact}[theorem]{Fact}

\newcommand{\wt}{\widetilde}

\newcommand{\R}{\mathbb{R}}

\renewcommand{\tilde}{\wt}

\renewcommand{\d}{\mathrm{d}}

\renewcommand{\Im}{\operatorname{Im}}
\newcommand{\T}{\mathcal{T}}
\newcommand{\Ab}{\mathbf{Ab}}

\DeclareMathOperator*{\E}{{\mathbb{E}}}
\DeclareMathOperator*{\Var}{{\bf {Var}}}
\DeclareMathOperator*{\Z}{\mathbb{Z}}
\DeclareMathOperator*{\C}{\mathbb{C}}

\DeclareMathOperator{\supp}{supp}
\DeclareMathOperator{\poly}{poly}

\DeclareMathOperator{\rank}{rank}

\DeclareMathOperator{\tr}{tr}

\newcommand{\todolow}[1]{\textbf{\color{yellow}[TODO: #1]}}
\renewcommand{\todolow}[1]{}

\definecolor{myblue}{HTML}{0099cc}

\newif\ifshowproof
\showprooffalse

\makeatletter
\newcommand*{\RN}[1]{\expandafter\@slowromancap\romantaumeral #1@}
\makeatother

\usepackage{lineno}

\input{libtheorems.tex}

\ifdefined\isarxiv
\else
\EventEditors{Kousha Etessami, Uriel Feige, and Gabriele Puppis}
\EventNoEds{3}
\EventLongTitle{50th International Colloquium on Automata, Languages, and Programming (ICALP 2023)}
\EventShortTitle{ICALP 2023}
\EventAcronym{ICALP}
\EventYear{2023}
\EventDate{July 10--14, 2023}
\EventLocation{Paderborn, Germany}
\EventLogo{}
\SeriesVolume{261}
\ArticleNo{38}
\fi
\begin{document}

\ifdefined\isarxiv
\begin{titlepage}
  \maketitle
  \begin{abstract}
  
In 2013, Marcus, Spielman, and Srivastava resolved the famous Kadison-Singer conjecture. It states that for $n$ independent random vectors $v_1,\cdots, v_n$ that have expected squared norm bounded by $\epsilon$ and are in the isotropic position in expectation, there is a positive probability that the determinant polynomial $\det(xI - \sum_{i=1}^n v_iv_i^\top)$ has roots bounded by $(1 + \sqrt{\epsilon})^2$. An interpretation of the Kadison-Singer theorem is that we can always find a partition of the vectors $v_1,\cdots,v_n$ into two sets with a low discrepancy in terms of the spectral norm (in other words, rely on the determinant polynomial). 

In this paper, we provide two results for a broader class of polynomials, the hyperbolic polynomials. Furthermore, our results are in two generalized settings:
\begin{itemize}
    \item The first one shows that the Kadison-Singer result requires a weaker assumption that the vectors have a bounded sum of hyperbolic norms.
    \item The second one relaxes the Kadison-Singer result's distribution assumption to the Strongly Rayleigh distribution.
\end{itemize}
To the best of our knowledge, the previous results only support determinant polynomials [Anari and Oveis Gharan'14, Kyng, Luh and Song'20]. It is unclear whether they can be generalized to a broader class of polynomials. In addition, we also provide a sub-exponential time algorithm for constructing our results.

  \end{abstract}
  \thispagestyle{empty}
\end{titlepage}
\else
 \maketitle
 \begin{abstract}
 
In 2013, Marcus, Spielman, and Srivastava resolved the famous Kadison-Singer conjecture. It states that for $n$ independent random vectors $v_1,\cdots, v_n$ that have expected squared norm bounded by $\epsilon$ and are in the isotropic position in expectation, there is a positive probability that the determinant polynomial $\det(xI - \sum_{i=1}^n v_iv_i^\top)$ has roots bounded by $(1 + \sqrt{\epsilon})^2$. An interpretation of the Kadison-Singer theorem is that we can always find a partition of the vectors $v_1,\cdots,v_n$ into two sets with a low discrepancy in terms of the spectral norm (in other words, rely on the determinant polynomial). 

In this paper, we provide two results for a broader class of polynomials, the hyperbolic polynomials. Furthermore, our results are in two generalized settings:
\begin{itemize}
    \item The first one shows that the Kadison-Singer result requires a weaker assumption that the vectors have a bounded sum of hyperbolic norms.
    \item The second one relaxes the Kadison-Singer result's distribution assumption to the Strongly Rayleigh distribution.
\end{itemize}
To the best of our knowledge, the previous results only support determinant polynomials [Anari and Oveis Gharan'14, Kyng, Luh and Song'20]. It is unclear whether they can be generalized to a broader class of polynomials. In addition, we also provide a sub-exponential time algorithm for constructing our results.

 \end{abstract}
\fi

\ifdefined\isarxiv
\newpage
\setcounter{page}{1}
\fi
\section{Introduction}

Introduced by \cite{ks59}, the Kadison-Singer problem was a long-standing open problem in mathematics. 
It was resolved by Marcus, Spielman, and Srivastrava in their seminal work \cite{mss15b}:
For any set of independent random vectors $u_1,\cdots,u_n$ such that each $u_i$ has finite support, and $u_1,\cdots,u_n$ are in isotropic positions in expectation, there is positive probability that $\sum_{i=1}^n u_i u_i^*$ has spectral norm bounded by $1 + O(\max_{i\in [n]} \|u_i\|)$. The main result of \cite{mss15b} is as follows:
\begin{theorem}[Main result of \cite{mss15b}]\label{thm:ks-original}
    Let $\epsilon > 0$ and let $v_1,\cdots,v_n \in \C^m$ be $n$ independent random vectors with finite support, such that $\E[\sum_{i=1}^n v_i v_i^*] = I$, and $\E[\|v_i\|^2] \leq \epsilon$, $\forall i\in [n]$. Then 
    \begin{align*}
        \Pr\left[ \Big\| \sum_{i\in [n]}v_iv_i^* \Big\| \leq (1+\sqrt{\epsilon})^2 \right] > 0.
    \end{align*}
\end{theorem}

The Kadison-Singer problem is closely related to discrepancy theory, which is an essential area in mathematics and theoretical computer science. A classical discrepancy problem is as follows: given $n$ sets over $n$ elements, can we color each element in red or blue such that each set has roughly the same number of elements in each color? More formally, for vectors $a_1,\dots, a_n\in \R^n$ with $\|a_i\|_\infty\leq 1$ and a coloring $s\in \{\pm 1\}^n$, the discrepancy is defined by $\mathrm{Disc}(a_1,\dots,a_n; s):=\|\sum_{i\in [n]}s_i a_i\|_\infty$. The famous Spencer's Six Standard Deviations Suffice Theorem \cite{spe85} shows that there exists a coloring with discrepancy at most $6\sqrt{n}$, which beats the standard Chernoff bound showing that a random coloring has discrepancy $\sqrt{n\log n}$. More generally, we can consider the ``matrix version'' of discrepancy: for matrices $A_1,\dots, A_n\in \R^{d\times d}$ and a coloring $s\in \{\pm 1\}^n$, 
\begin{align*}
    \mathrm{Disc}(A_1,\dots,A_n; s):=\Big\|\sum_{i\in [n]}s_i A_i\Big\|.
\end{align*}
Theorem~\ref{thm:ks-original} is equivalent to the following discrepancy result for rank-1 matrices: 
\begin{theorem}[\cite{mss15b}]\label{thm:KS}
Let $u_1, \dots, u_n \in \C^m$ and
suppose $\max_{i \in [n]} \| u_i u_i^*\| \leq \epsilon$ and $ \sum_{i=1}^n  u_i u_i^* = I$.
Then, 
\begin{align*}
    \min_{s\in \{\pm 1\}^n}~\mathrm{Disc}(u_1u_1^*,\dots,u_nu_n^*; s) \leq O(\sqrt{\epsilon}).
\end{align*}
\end{theorem}

In other words, the minimum discrepancy of rank-1 isotropic matrices is bounded by $O(\sqrt{\epsilon})$, where $\epsilon$ is the maximum spectral norm. %
This result also beats the matrix Chernoff bound \cite{tro15}, which shows that a random coloring for matrices has discrepancy $O(\sqrt{\epsilon \log d})$.
The main techniques in \cite{mss15b} are the method of interlacing polynomials and the barrier methods developed in \cite{mss15a}.

Several generalizations of the Kadison-Singer-type results, which have interesting applications in theoretical computer science, have been established using the same technical framework as described in \cite{mss15b}.
In particular, Kyng, Luh, and Song \cite{kls20} provided a ``four derivations suffice'' version of Kadison-Singer conjecture: Instead of assuming every independent random vector has a bounded norm, the main result in \cite{kls20} only requires that the sum of the squared spectral norm is bounded by $\sigma^2$, and showed a discrepancy bound of $4\sigma$:

\begin{theorem}[\cite{kls20}]\label{thm:kls20}
Let $u_1, \dots, u_n \in \C^m$ and $\sigma^2 = \| \sum_{i=1}^n (u_i u_i^{*})^2 \|$. 
Then, we have
\begin{align*}
\Pr_{\xi \sim \{ \pm 1 \}^n } \left[ \Big\| \sum_{i=1}^n \xi_i u_i u_i^* \Big\| \leq 4 \sigma \right] >0.
\end{align*} 
\end{theorem}

This result was recently applied by \cite{lz20} to  approximate solutions of generalized network design problems.

Moreover, Anari and Oveis-Gharan \cite{ag14} generalized the Kadison-Singer conjecture into the setting of real-stable polynomials. Instead of assuming  the random vectors are independent, \cite{ag14} assumes that the vectors are sampled from any homogeneous strongly Rayleigh distribution with bounded marginal probability, have bounded norm, and are in an isotropic position:

\begin{theorem}[\cite{ag14}]\label{thm:ag14}
Let $\mu$ be a homogeneous strongly Rayleigh probability distribution on $[n]$ such that the marginal probability of each element is at most $\epsilon_1$, and let $u_1, \cdots, u_n \in \R^m$ be vectors in an isotropic position, $\sum_{i=1}^n u_i u_i^* = I$, such that $\max_{i\in [n]} \| u_i \|^2 \leq \epsilon_2$. Then
\begin{align*}
    \Pr_{S \sim \mu} \left[ \Big\| \sum_{i \in S} u_i u_i^* \Big\| \leq 4 (\epsilon_1+ \epsilon_2) + 2 (\epsilon_1+\epsilon_2)^2 \right] >0.
\end{align*}
\end{theorem}

Theorem \ref{thm:ag14} has a direct analog in spectral graph theory: Given any (weighted) connected graph $G = (V,E)$ with Laplacian $L_G$. For any edge $e=(u,v)\in E$, define the vector corresponding to $e$ as $v_e = L_G^{\dagger/2} (\mathbf{1}_u - \mathbf{1}_v)$ (here $L_G^{\dagger}$ is the Moore-Penrose inverse). Then the set of $\{v_e : e\in E\}$ are in isotropic position, and $\|v_e\|^2$ equals to the graph effective resistance with respect to $e$. Also, any spanning tree distribution of the edges in $E$ is homogeneous strongly Rayleigh. It follows from Theorem \ref{thm:ag14} that any graph with bounded maximum effective resistance has a spectrally-thin spanning tree \cite{ag14}. Moreover, \cite{ao15} provided an exciting application to the asymmetric traveling salesman problem and obtained an $O(\log\log n)$-approximation.

Another perspective of generalizing the Kadison-Singer theorem is to study the discrepancy with respect to a more general norm than the spectral norm, which is the largest root of a determinant polynomial. A recent work by Br\"anden \cite{b18} proved a high-rank version of Theorem~\ref{thm:KS} for \emph{hyperbolic} polynomial, which is a larger class of polynomials including the determinant polynomial.  
Moreover, the hyperbolic norm on vectors is a natural generalization of the matrix spectral norm. 
\ifdefined\isarxiv
(We will introduce hyperbolic polynomials in Section~\ref{sec:our_result}.)
\else
We will introduce hyperbolic polynomials in the full version of our paper.
\fi
From this perspective, it is very natural to ask:
\begin{center}
    \emph{Can we also extend Theorem~\ref{thm:kls20} and Theorem~\ref{thm:ag14} to a more general class of polynomials, e.g., hyperbolic polynomials?}
\end{center}

\subsection{Our results}\label{sec:our_result}

In this work, we provide an affirmative answer by generalizing both Theorem~\ref{thm:kls20} and Theorem~\ref{thm:ag14} into the setting of hyperbolic polynomials. %
Before stating our main results, we first introduce some basic notation of hyperbolic polynomials below.

Hyperbolic polynomials form a broader class of polynomials that encompasses determinant polynomials and homogeneous real-stable polynomials.
An $m$-variate, degree-$d$ homogeneous polynomial $h \in \mathbb{R}[x_1, \cdots, x_m]$ is \emph{hyperbolic} with respect to a direction $e \in \mathbb{R}^m$ if the univariate polynomial $t \mapsto h(te - x)$ has only real roots for all $x \in \mathbb{R}^m$.
\ifdefined\isarxiv
(See Appendix~\ref{sec:example} for some examples of hyperbolic/real-stable polynomials.)
\fi
The set of $x \in \mathbb{R}^m$ such that all roots of $h(t\mathrm{e} - x)$ are non-negative (or strictly positive) is referred to as the hyperbolicity cone $\Gamma^h_{+}(e)$ (or $\Gamma^h_{++}(e)$).
It is a widely recognized result \cite{b10notes} that any vector $x$ in the open hyperbolicity cone $\Gamma^h_{++}(e)$ is itself hyperbolic with respect to the polynomial $h$ and have the same hyperbolicity cone as $e$, meaning that $\Gamma^h_{++}(e) = \Gamma^h_{++}(x)$.
Therefore, the unique hyperbolicity cone of $h$ can simply be expressed as $\Gamma^h_{+}$. 

The hyperbolic polynomials have similarities to determinant polynomials of matrices, as they both can be used to define trace, norm, and eigenvalues.
Given a hyperbolic polynomial $h \in \mathbb{R}[x_1, \cdots, x_m]$ and any vector $e \in \Gamma^h_{++}$,
we can define a norm with respect to $h(x)$ and $e$ as follows: for any $x\in \R^m$, its \emph{hyperbolic norm} $\|x\|_h$ is equal to the largest root (in absolute value) of the linear restriction polynomial $h(te - x)\in \R[t]$. 
Similar to the eigenvalues of matrices, we define the \emph{hyperbolic eigenvalues} of $x$ to be the $d$ roots of $h(te-x)$, denoted by $\lambda_1(x)\geq \cdots\geq \lambda_d(x)$. 
We can also define the \emph{hyperbolic trace} and the \emph{hyperbolic rank}:
\begin{align*}
    \tr_h[x]:=\sum_{i=1}^d \lambda_i(x),~~\text{and}~~\rank(x)_h:=|\{i\in[d]: \lambda_i(x)\ne 0\}|.
\end{align*}

Recall that both Theorem~\ref{thm:kls20} and Theorem~\ref{thm:ag14} upper-bound the spectral norm of the sum $\|\sum_{i=1}^n \xi_i v_i v_i^\top\|$. In the setting of hyperbolic polynomials, we should upper bound the hyperbolic norm $\|\sum_{i=1}^n \xi_i v_i\|_h$ for vectors $v_1,\dots,v_n$ in the hyperbolicity cone, which is the set of vectors with all non-negative hyperbolic eigenvalues.

Our main results are as follows:

\begin{theorem}[Main Result I, informal
\ifdefined\isarxiv
 statement of Theorem~\ref{thm:kls20_ours_formal},
\fi
hyperbolic version of Theorem 1.4, \cite{kls20}]\label{thm:kls20_ours}
Let $h\in \R[x_1,\dots,x_m]$ denote a hyperbolic polynomial in direction $e\in \R^m$. Let $v_1, \dots, v_n \in \Gamma_{+}^h$ be $n$ vectors in the closed hyperbolicity cone. Let $\xi_1,\dots,\xi_n$ be $n$ independent random variables with finite supports and $\E[\xi_i]=\mu_i$ and $\Var[\xi_i]=\tau_i^2$. Suppose $\sigma :=  \| \sum_{i=1}^n \tau_i^2\tr_h[v_i] v_i  \|_h$.
Then there exists an assignment $(s_1,\dots,s_n)$ with $s_i$ in the support of $\xi_i$ for all $i\in [n]$, such that
\begin{align*}
\Big\| \sum_{i=1}^n (s_i-\mu_i) v_i \Big\|_h \leq 4 \sigma.
\end{align*} 
\end{theorem}

We remark that Theorem~\ref{thm:kls20_ours} does not require the isotropic position condition of $v_1, \cdots, v_n$ as in \cite{b18}. In addition, we only need the sum of $\tr_h[v_i]v_i$'s hyperbolic norm to be bounded, while \cite{b18}'s result requires each vector's trace to be bounded individually. %

We would also like to note that the class of hyperbolic polynomials is much broader than that of determinant polynomials, which were used in the original Kadison-Singer-type theorems.  Lax conjectured in \cite{lax57} that every 3-variate hyperbolic/real-stable polynomial could be represented as a determinant polynomial, this was later resolved in \cite{hv07,lpr05}. However, the Lax conjecture is false when the number of variables exceeds $3$, as demonstrated in \cite{bra11,bvy14} with counterexamples of hyperbolic/real-stable polynomials $h(x)$ for which even $(h(x))^k$ cannot be represented by determinant polynomials for any $k > 0$.

Our second main result considers the setting where the random vectors are not independent, but instead, sampled from a strongly Rayleigh distribution. We say a distribution $\mu$ over the subsets of $[n]$ is \emph{strongly Rayleigh} if its generating polynomial $g_\mu(z):=\sum_{S\subseteq [n]}\mu(S)z^S\in \R[z_1,\dots,z_n]$ is a \emph{real-stable polynomial}, which means $g_\mu(z)$ does not have any root in the upper-half of the complex plane, i.e., $g_\mu(z)\ne 0$ for any $z\in \C^n$ with $\Re(z)\succ 0$. 
\begin{theorem}[Main Result II, informal 
\ifdefined\isarxiv
statement of Theorem \ref{thm:ag14_ours_formal} 
\fi
hyperbolic version of Theorem 1.2, \cite{ag14}]\label{thm:ag14_ours}
Let $h\in \R[x_1,\dots,x_m]$ denote hyperbolic polynomial in direction $e \in \R^m$.
Let $\mu$ be a homogeneous strongly Rayleigh probability distribution on $[n]$ such that the marginal probability of each element is at most $\epsilon_1$. 

Suppose $v_1, \cdots, v_n \in \Gamma_{+}^h$ are in the hyperbolicity cone of $h$ such that $\sum_{i=1}^n v_i  = e$, and
for all $i \in [n]$, $ \| v_i \|_h \leq \epsilon_2$.
Then there exists $S \subseteq [n]$ in the support of $\mu$, such that
\begin{align*}
     \Big\| \sum_{i \in S} v_i \Big\|_h \leq 4 (\epsilon_1+ \epsilon_2) + 2 (\epsilon_1+\epsilon_2)^2 .
\end{align*}
\end{theorem}

It is worth mentioning that the previous paper \cite{kls20,ag14} focused on the determinant polynomial, leaving the question of whether their techniques could be extended to the hyperbolic/real-stable setting unresolved. In our paper, we address this gap by developing new techniques specifically tailored to hyperbolic polynomials.

In addition, we follow the results from \cite{aoss18} and give an algorithm that can find the approximate solutions of both Theorem \ref{thm:kls20_ours} and Theorem \ref{thm:ag14_ours} in time sub-exponential to $m$:

\begin{proposition}[Sub-exponential algorithm for Theorem \ref{thm:kls20_ours}, informal 
\ifdefined\isarxiv
statement of Corollary \ref{cor:subexpalg-kls20-ours} 
\fi]\label{prop:subexpalg-kls20-ours-informal}
    Let $h\in \R[x_1,\dots,x_m]$ denote a hyperbolic polynomial with direction $e\in \R^m$. Let $v_1, \dots, v_n \in \Gamma_{+}^h$ be $n$ vectors in the hyperbolicity cone $\Gamma_{+}^h$ of $h$. Suppose $\sigma = \| \sum_{i=1}^n \tr_h[v_i] v_i \|_h$.

    Let $\mathcal{P}$ be the interlacing family used in the proof of Theorem \ref{thm:ag14_ours}.
    Then there exists an sub-exponential time algorithm $\mathrm{KadisonSinger}(\delta, \mathcal{P})$, such that for any $\delta > 0$, it returns a sign assignment $(s_1,\cdots,s_n)\in \{\pm 1\}^n$ satisfying
    \begin{align*}
     \Big\| \sum_{i=1}^n s_i u_i \Big\|_h \leq 4(1+\delta) \sigma.
    \end{align*} 
\end{proposition}

\begin{proposition}[Sub-Exponential algorithm for Theorem \ref{thm:ag14_ours}, informal
\ifdefined\isarxiv
statement of Corollary \ref{cor:subexpalg-ag14}
\fi
]\label{prop:subexpalg-ag14-informal}
    Let $h\in \R[x_1,\dots,x_m]$ denote a hyperbolic polynomial in direction $e\in \R^m$. 
    Let $\mu$ be a homogeneous strongly Rayleigh probability distribution on $[n]$ such that the marginal probability of each element is at most $\epsilon_1$, and let $v_1, \cdots, v_n \in \Gamma_{+}^h$ be $n$ vectors such that $\sum_{i=1}^n v_i  = e$, and for all $i \in [n]$, $\| v_i \|_h \leq \epsilon_2$.

    Let $\mathcal{Q}$ be the interlacing family used in the proof of Theorem \ref{thm:ag14_ours}.
    Then there exists an sub-exponential time algorithm $\mathrm{KadisonSinger}(\delta, \mathcal{Q})$, such that for any $\delta > 0$, it returns a set $S$ in the support of $\mu$ satisfying
    \begin{align*}
     \Big\| \sum_{i\in S} u_i \Big\|_h \leq (1+\delta) \cdot  \left(4(\epsilon_1+ \epsilon_2) + 2 (\epsilon_1+\epsilon_2)^2\right).
    \end{align*} 
\end{proposition}

\section{Related work}

\paragraph*{Real-Stable Polynomials}
Real-stability is an important property for multivariate polynomials. 
In \cite{bb09}, the authors used the real-stability to give a unified framework for Lee-Yang type problems in statistical mechanics and combinatorics. Real-stable polynomials are also related to the permanent. Gurvits \cite{gur07} proved the Van der Waerden conjecture, which conjectures that the permanent of $n$-by-$n$ doubly
stochastic matrices are lower-bounded by $n!/n^n$, via the capacity of real-stable polynomials. Recently, \cite{gl21} improved the capacity lower bound for real-stable polynomials, which has applications in matrix scaling and metric TSP. In addition, real-stable polynomials are an important tool in solving many counting and sampling problems \cite{ns16, aor15, ao17, sv17, aos16,amg18,aov18,algv19,algvv21}.

\paragraph*{Hyperbolic Polynomials}
Hyperbolic polynomial was originally defined to study the stability of partial differential equations \cite{g51, h83, k95}. In theoretical computer science, G{\"u}ler \cite{g97} first introduced hyperbolic polynomial for optimization (hyperbolic programming), which is a generalization of LP and SDP. Later, a few algorithms \cite{r06,mt14, rs14, ren16,np18, ren19} were designed for hyperbolic programming. On the other hand, a significant effort has been put into the equivalence between hyperbolic programming and SDP, which is closely related to the ``Generalized Lax Conjecture'' (which conjectures that every hyperbolicity cone is spectrahedral) and its variants \cite{hv07,lpr05, b14,kpv15,sau18,ami19,rrsw19}. 

\paragraph*{Strongly Rayleigh Distribution} The strongly Rayleigh distribution was introduced by \cite{bbl09}. 
The authors also proved numerous basic properties of strongly Rayleigh distributions, including negative association, and closure property under operations such as conditioning, product, and restriction to a subset. \cite{pp14} proved a concentration result for Lipschitz functions of strongly Rayleigh variables. 
\cite{ks18} showed a matrix concentration for strongly Rayleigh random variables, which implies that adding a small number of uniformly random spanning trees gives a graph spectral sparsifier. 

Strongly Rayleigh distribution also has many algorithmic applications. \cite{aor15} exploited the negative dependence property of homogeneous strongly Rayleigh distributions, and designed efficient algorithms for generating approximate samples from Determinantal Point Process using Monte Carlo Markov Chain.
The strongly Rayleigh property of spanning tree distribution is a key component for improving the approximation ratios of TSP \cite{kko20a,kko20b} and $k$-edge connected graph problem \cite{kkoz21}.

\paragraph*{Other generalizations of the Kadison-Singer-type results}

The upper bound of the rank-one Kadison-Singer theorem was improved by \cite{bcms19,rl20}. \cite{ab20} further extended \cite{rl20}'s result to prove a real-stable version of Anderson’s paving conjecture. However, they used a different norm for real-stable polynomials, and hence their results and ours are incomparable. In the high-rank case, \cite{c16} also proved a Kadison-Singer result for high-rank matrices. \cite{sz20} relaxed \cite{b18}'s result to the vectors in sub-isotropic position. In addition, they proved a hyperbolic Spencer theorem for constant-rank vectors. 

Another direction of generalizing the Kadison-Singer-type result is to relax the $\{+1, -1\}$-coloring to $\{0,1\}$-coloring, which is called the one-sided version of Kadison-Singer problem in \cite{w04}. More specifically, given $n$ isotropic vectors $v_1,\dots,v_n\in \R^m$ with norm $\frac{1}{\sqrt{N}}$, the goal is to find a subset $S\subset [n]$ of size $k$ such that $\|\sum_{i\in S}v_iv_i^\top\|\leq \frac{k}{n}+O(1/\sqrt{N})$. Unlike the original Kadison-Singer problem, Weaver \cite{w04} showed that this problem can be solved in polynomial time. Very recently, Song, Xu and Zhang \cite{sxz22} improved the time complexity of the algorithm via an efficient inner product search data structure.
\paragraph*{Applications of Kadison-Singer Problem} There are many interesting results developed from the Kadison-Singer theorem.
In spectral graph theory, \cite{ho14} exploited the same proof technique of interlacing families to show a sufficient condition of the spectrally thin tree conjecture.
\cite{ag14} used the strongly-Rayleigh extension of Kadison-Singer theorem to show a weaker sufficient condition.
Based on this result, \cite{ao15} showed that any $k$-edge-connected graph has an $O(\frac{\log\log (n)}{k})$-thin tree, and gave a $\poly(\log\log(n))$-integrality gap of the asymmetric TSP. \cite{mss18,coh16} used the Kadison-Singer theorem to construct bipartite Ramanujan graphs of all sizes and degrees. 
In the network design problem, \cite{lz20} exploited the result in \cite{kls20}, and built a spectral rounding algorithm for the general network design convex program, which has applications in weighted experimental design, spectral network design, and additive spectral sparsifier.

\section{Proof Overview}

\subsection{Hyperbolic Deviations}

In this section, we will sketch the proof of our hyperbolic generalization of the Kadison-Singer theorem (Theorem~\ref{thm:kls20_ours}). 
\ifdefined\isarxiv
\else
Details of the proof are deferred to the full version of the paper.
\fi
We will use the same strategy as the original Kadison-Singer theorem (Theorem \ref{thm:ks-original}) in \cite{mss15a,mss15b}, following three main technical steps.

For simplicity, we assume that the random variables $\xi_1,\dots,\xi_n\in \{\pm 1\}$ are independent Rademacher random variables, i.e., $\Pr[\xi_i=1]=\frac{1}{2}$ and $\Pr[\xi_i=-1]=\frac{1}{2}$ for all $i\in [n]$. 

To generalize the Kadison-Singer statement into the hyperbolic norm, one main obstacle is to define the variance of the hyperbolic norm of the sum of random vectors $\sum_{i=1}^n \xi_i v_i$. In the determinant polynomial case, each $v_i$ corresponds to a rank-1 matrix $u_iu_i^*$, and it is easy to see that the variance of the spectral norm is $\|\sum_{i=1}^n (u_iu_i^*)^2\|$. However, there is no analog of ``matrix square'' in the setting of hyperbolic/real-stable polynomials. Instead, we define the \emph{hyperbolic variance}: $$\left\|\sum_{i=1}^n \tr_h[v_i]v_i\right\|_h$$ in terms of the hyperbolic trace, and show that \emph{four hyperbolic deviations suffice}.

\paragraph*{Defining interlacing family of characteristic polynomials.}
In the first step, we construct a family of \emph{characteristic polynomials} $\{p_s : s \in \{\pm 1\}^t, t \in \{0, \cdots, n\}\}$ as follows:
For each ${\bf s}\in \{\pm 1\}^n$, define the leaf-node-polynomial:
\begin{align*}
    p_{{\bf s}}(x) := \left(\prod_{i=1}^n p_{i, s_i}\right)\cdot h\left(x e + \sum_{i=1}^n s_i  v_i  \right) \cdot h\left(x e - \sum_{i=1}^n s_i  v_i  \right),
\end{align*}
and for all $\ell \in \{0,\dots,n-1\}$, ${\bf s'}\in \{\pm 1\}^\ell$, we construct an inner node with a polynomial that corresponds to the bit-string $\mathbf{s}'$:
\begin{align*}
    p_{{\bf s'}}(x) := \sum_{{\bf t}\in \{\pm 1\}^{n-\ell}} p_{({\bf s'}, {\bf t})}(x).
\end{align*}
where $({\bf s'}, {\bf t}) \in \{\pm 1\}^n$ is the bit-string concatenated by ${\bf s'}$ and ${\bf t}$.

We will then show that the above family of characteristic polynomials forms an \emph{interlacing family} 
\ifdefined\isarxiv
(see Lemma~\ref{lem:interlacing-family-kls20} for detail)
\fi
. By basic properties of interlacing family, we can always find a leaf-root-polynomial $p_s$ (where $s \in \{\pm 1\}^n$) whose largest root is upper bounded by the largest root of the top-most polynomial. 
\begin{align*}
    p_{\emptyset}(x) = \E_{ \xi_1, \cdots, \xi_n } \left[ h \Big( x e +  \sum_{i=1}^n \xi_i v_i \Big) \cdot h \Big( x e - \sum_{i=1}^n \xi_i v_i \Big) \right].
\end{align*}
(we call $p_{\emptyset}$ to be the \emph{mixed characteristic polynomial}). 
Notice that by rewriting the largest root of $p_s$ to be the expected hyperbolic norm of $\sum_{i=1}^n s_i v_i$, we get that
\begin{align}\label{eq:kls_root_cond_intro}
    \lambda_{\max}(p_\emptyset) = \left\|\sum_{i=1}^n s_i v_i \right\|_h. 
\end{align}
\ifdefined\isarxiv
(See Corollary~\ref{cor:kls-interlacing} for a formal statement).
\fi

Also, we will take $s \in \{\pm 1\}^n$ as the corresponding sign assignment in the main theorem (Theorem \ref{thm:kls20_ours})
It then suffices to upper-bound the largest root of the mixed characteristic polynomial.

\paragraph*{From mixed characteristic polynomial to multivariate polynomial.}

In the second step, we will show that the mixed characteristic polynomial that takes the average on $n$ random variables
\begin{align*}
    p_{\emptyset}(x) = \E_{ \xi_1, \dots, \xi_n } \left[ h \Big( x e +  \sum_{i=1}^n \xi_i v_i \Big) \cdot h \Big( x e - \sum_{i=1}^n \xi_i v_i \Big) \right]
\end{align*}
is equivalent to a polynomial with $n$ extra variables $z_1, \cdots, z_n$:
\begin{align}\label{eq:kls_multivariate_intro}
    \prod_{i=1}^n \Big(1- \frac{1}{2} \frac{\partial^2}{ \partial z_i^2 }\Big)\Bigg|_{z=0} \left(h \Big( xe + \sum_{i=1}^n z_i v_i \Big)\right)^2.
\end{align}
\ifdefined\isarxiv
(See Lemma \ref{lem:kls_expectation_hyperbolic} for more detail.)
\fi
Thus, we can reduce the upper bound of $\chi_{\max}(p_{\emptyset})$ to an upper bound of the largest root in \eqref{eq:kls_multivariate_intro}.
The latter turns out to be easier to estimate with the help of a barrier argument \cite{mss15b}.

To show such equivalence holds, we use induction on the random variables $\xi_1,\dots,\xi_n$. More specifically, we start from $\xi_1$ and are conditioned on any fixed choice of $\xi_2,\dots,\xi_n$. We prove that taking expectation over $\xi_1$ is equivalent to applying the operator $(1-\frac{\partial^2}{\partial z_1^2})$ to the polynomial
\begin{align*}
    \left(h(xe + z_1 v_1 + \sum_{i=2}^n \xi_i v_i)\right)^2
\end{align*}
and setting $z_1=0$.
Here we use the relation between expectation and the second derivatives: for any Rademacher random variable $\xi$,
\begin{align*}
    \E_{\xi} [ h(x_1 - \xi v) \cdot h(x_2 + \xi v) ] = \left(1 - \frac{1}{2}\frac{\d ^2}{\d t^2}\right)\Bigg|_{t=0} h(x_1+t v) h(x_2+t v).
\end{align*}
Repeating this process and removing one random variable at a time. After $n$ iterations, we obtain the desired multivariate polynomial.

We also need to prove the real-rootedness of the multivariate polynomial 
(Eqn.~\eqref{eq:kls_multivariate_intro}). We first consider an easy case where $h$ itself is a real-stable polynomial, as in the determinant polynomial case. Then the real-rootedness easily follows from the closure properties of the real-stable polynomial 
\ifdefined\isarxiv
(see Fact~\ref{fac:closure_real_stable})
\fi. More specifically, we can show that $(h(xe+\sum_{i=1}^nz_iv_i))^2$ is also a real-stable polynomial. Furthermore, applying the operators $(1-\frac{1}{2}\frac{\partial^2}{\partial z_i^2})$ and restricting $z=0$ preserve the real-stability. Therefore, the multivariate polynomial is a univariate real-stable polynomial, which is equivalent to being real-rooted.

Next, we show that when $h$ is a hyperbolic polynomial, the multivariate polynomial (Eqn.~\eqref{eq:kls_multivariate_intro}) is also real-rooted. 
our approach is to show that the linear restriction of $h$: $h(xe+\sum_{i=1}^n z_iv_i)$ is a real-stable polynomial in $\R[x,z_1,\dots,z_n]$. A well-known test for real-stability is that if for any $a\in \R_{>0}^{n+1}, b\in \R^{n+1}$, the one-dimensional restriction $p(at+b)\in \R[t]$ is non-zero and real-rooted, then $p(x)$ is real-stable. We test $h(xe+\sum_{i=1}^n z_iv_i)$ by restricting to $at+b$, and get the following polynomial:
\begin{align*}
   h \Big( (a_1 e + \sum_{i=1}^n a_{i+1}v_i)t + y \Big)\in \R[t],
\end{align*}
where $y$ is a fixed vector depending on $b$. Since $a_i>0$ for all $i\in [n+1]$ and $e,v_1,\dots,v_n$ are vectors in the hyperbolicity cone, it implies that the vector $a_1 e + \sum_{i=1}^n a_{i+1}v_i$ is also in the hyperbolicity cone. Then, by the definition of hyperbolic polynomial, we immediately see that $h( (a_1 e + \sum_{i=1}^n a_{i+1}v_i)t + y)$ is real-rooted for any $a\in \R_{>0}^{n+1}$ and $b\in \R^{n+1}$. Hence, we can conclude that the restricted hyperbolic polynomial $h(xe+\sum_{i=1}^n z_iv_i)$ is real-stable and the remaining proof is the same as the real-stable case. 

\paragraph*{Applying barrier argument. }
Finally, we use \emph{barrier argument} to find an ``upper barrier vector'' whose components lie above any roots of  multivariate polynomial can take. In particular, we consider the multivariate polynomial $P(x,z)=(h(xe+\sum_{i=1}^n z_iv_i))^2$.
Define the \emph{barrier function} of any variable $i \in [n]$ as the following:
\begin{align*}
    \Phi^i_P ( \alpha(t), - \delta)
= \frac{ \partial_{z_i} P(x,z) }{ P(x,z) } \Big|_{ x = \alpha(t), z = -\mathbf{\delta} },
\end{align*}
where $\delta \in \R^n$ where $\delta_i=t\tr_h[v_i]$ for $i\in [n]$ and $\alpha(t)>t$ is a parameter that depends on $t$. 

As a warm-up, consider the case when $\sigma=1$ and assuming $\|\sum_{i=1}^n \tr_h[v_i]v_i\|_h\leq 1$. It is easy to show that $(\alpha(t), -\delta)$ is an upper barrier of $P$, from the linearity of the hyperbolic eigenvalues and the assumption.
Next, we upper-bound the barrier function's value at $(\alpha(t), -\delta)$. 
When $h$ is a determinant polynomial, this step is easy because the derivative of $\log \det$ is the trace of the matrix. For a general hyperbolic polynomial, we will rewrite the partial derivative $\partial_{z_i}$ as a directional derivative $D_{v_i}$ and get
\begin{align*}
    \Phi^i_P ( \alpha(t), - \delta )=2\cdot \frac{(D_{v_i}h)(\alpha e-te + t(e - \sum_{j=1}^n \tr_h[v_j] v_j))}{h(\alpha e -te + t(e - \sum_{j=1}^n \tr_h[v_j] v_j))}.
\end{align*}
We observe that our assumption $\|\sum_{i=1}^n \tr_h[v_i]v_i\|_h\leq 1$ implies that $e - \sum_{j=1}^n \tr_h[v_j] v_j\in \Gamma_+^h$. By the concavity of the function $\frac{h(x)}{D_{v_i} h(x)}$ in the hyperbolicity cone, we can prove that
\begin{align*}
    \Phi^i_P ( \alpha(t), - \delta ) \leq \frac{ 2\tr_h[v_i] }{ \alpha(t) - t }.
\end{align*}
Now, we can apply the barrier update lemma in \cite{kls20}
\ifdefined\isarxiv 
(Lemma~\ref{lem:real_stable_cone})
\fi 
with $\alpha(t) =2t=4$ to show that 
\begin{align*}
    \Phi^j_{(1 - \frac{1}{2} \partial_{z_i}^2)P} (4, -\mathbf{\delta} + \delta_i {\bf 1}_i) \le \Phi_P^j(4, -\mathbf{\delta}).
\end{align*}
In other words, the partial differential operator $(1 - \frac{1}{2}\partial_{z_i}^2)$ shifts the upper-barrier by $(0, \cdots, 0, \delta_i , 0, \cdots, 0)$. Using induction for the variables $\delta_1, \cdots, \delta_n$, we can finally finally get an upper-barrier of
\begin{align*}
    (4,-\delta+\sum_{i=1}^n \delta_i {\bf 1}_i) = (4,0,\dots,0),
\end{align*}
which implies that $(4,0,\dots,0)$ is above the roots of 
\begin{align}
    \prod_{i=1}^n \Big( 1 - \frac{1}{2} \frac{ \partial^2 }{ \partial z_i^2 } \Big)  \left( h \Big( x e + \sum_{i=1}^n z_i\tau_i v_i \Big) \right)^2 
    \label{eq:mix_polynomial_kls}
\end{align} 
\ifdefined\isarxiv 
(See Lemma \ref{lem:kls_bound_root} for detail).
\fi

A challenge in this process is ensuring that the barrier function remains nonnegative. To achieve this, we use the multidimensional convexity of the hyperbolic barrier function as established in \cite{tao13} 
\ifdefined\isarxiv (Lemma~\ref{lem:multi_phi_convex}) \fi
. For cases where $\sigma \ne 1$, this requirement is satisfied through a simple scaling argument.

Combining the above three steps together, we can prove that $\Pr_{\xi_1,\cdots,\xi_n} [\| \sum_{i=1}^n \xi_i v_i \|_h \leq 4 \sigma ] >0$ for vectors $v_1,\dots,v_n$ in the hyperbolicity cone with $\| \sum_{i=1}^n \tr_h[v_i] v_i  \|_h=\sigma^2$.
\ifdefined\isarxiv
The complete proof can be found in Appendix \ref{sec:kls}.
\fi

\subsection{Generalization to Strongly Rayleigh Distributions}

Our main technical contribution to Theorem \ref{thm:ag14_ours} is a more universal and structured method to characterize the mixed characteristic polynomial.
Define the mixed characteristic polynomial as
\begin{align}\label{eq:mix-characteristic-ag14}
    q_S(x) = \mu(S) \cdot h \left( x e - \sum_{i \in S} v_i \right).
\end{align}
we want to show that it is equivalent to the restricted multivariate polynomial: \begin{align}\label{eq:multivariate-ag14}
    \prod_{i=1}^n (1- \frac{1}{2} \frac{\partial^2}{ \partial z_i^2 })\Big(h(xe+\sum_{i=1}^nz_iv_i)g_\mu(x{\bf 1}+z)\Big)\Bigg|_{z=0}  \in \R[x,z_1,\cdots,z_n].
\end{align}

Although Eqn. (\ref{eq:mix-characteristic-ag14}) and Eqn. (\ref{eq:multivariate-ag14}) are the hyperbolic generalization of \cite{ag14}, we are unable to apply the previous techniques. This is because \cite{ag14} computes the mixed characteristic polynomial explicitly, which heavily relies on the fact that the characteristic polynomial is a determinant. 
It is unclear how to generalize this method to hyperbolic/real-stable characteristic polynomials.

The key step in \cite{ag14} is to show the following equality between mixed characteristic polynomial and multivariate polynomial:
\begin{align*}
    & x^{d_{\mu} - d } \cdot \E_{S \sim \mu} \left[  \det \left( x^2 I - \sum_{i \in S} 2 v_i v_i^\top \right) \right] \\
    = & \prod_{i=1}^n (1-\partial_{z_i}^2)  \left( g_{\mu}( x {\bf 1} + z ) \cdot \det(x I + \sum_{i=1}^n z_i v_i v_i^\top) \right) \Bigg|_{z = {\bf 0}}
\end{align*}
where $d_{\mu}$ is the degree of the homogeneous strongly-Rayleigh distribution $\mu$ (i.e. the degree of $g_{\mu}$), and $m$ is the dimension of $v_i$.

Then they expand the right-hand side to get: 
\begin{align*}
    {\rm RHS}
    = & ~ \sum_{k=0}^m (-1)^k x^{d_{\mu} + m -2 k} \sum_{S \in { \binom{[n]}{k} }} \Pr_{T \sim \mu } [ S \subseteq T ] \cdot \sigma_k ( \sum_{i \in S} 2v_i v_i^\top ) \\
    = & ~ x^{d_{\mu} - m } \cdot \E_{S \sim \mu} \left[  \det \left( x^2 I - \sum_{i \in S} 2 v_i v_i^\top \right) \right] = {\rm LHS},
\end{align*}
where $\sigma_k(M)$ equals to the sum of all $k\times k$ principal minors of $M\in \R^{m\times m}$. The first step comes from expanding the product $\prod_{i=1}^n (1-\partial^2 z_i)$, and the second step comes from that
\begin{align*}
    \det(x^2I - \sum_{i=1}^n v_iv_i^\top) = \sum_{k=0}^m (-1)^{2k} x^{2m-2k} \sum_{S \in \binom{[n]}{k}} \sigma_k( \sum_{i\in S} v_i v_i^\top).
\end{align*}

The naive generalization of a technique to hyperbolic/real-stable polynomial $h$ faces challenges. One such challenge is the absence of an explicit form for $h$, unlike in the case of $h=\det$ where the determinant can be expressed as a combination of minors. This lack of a well-defined minor presents difficulty in rewriting the hyperbolic/real-stable polynomial. 
To tackle this issue, we devised a new and structured proof that relies on induction, offering a novel solution to this problem.

\paragraph*{Inductive step.}
We first rewrite the expectation over the Strongly-Rayleigh distribution $T \sim \mu$ as follows:
\begin{align*}
    x^{d_{\mu}} \cdot 2^{-n} \cdot \E_{T \sim \mu} [ h(x e - \sum_{i \in T} v_i) ] 
    = & ~ \frac{1}{2} \E_{\xi_2, \cdots, \xi_n \sim \{0,1\}^{n-1}} \Big[ (1-\partial_{z_1}) h(x_2 + z_1 v_1) x \partial_{z_1} g_2(x+z_1)  \\
    & ~~~ + h(x_2) (1-x \partial_{z_1}) g_2(x+z_1)\Big|_{z_1=0} \Big]
\end{align*}
where $g_2$ is defined as 
\begin{align*}
     g_2(t)  := & ~ x^{\sum_{i=2}^n \xi_i}\cdot \\
     & ~~~ \prod_{i=2}^n \Big(\xi_i \partial_{z_i} +  (1-\xi_i)(1-x\partial_{z_i})\Big) g_{\mu}(t,x+z_2,x+z_3,\cdots,x+z_n)\Big|_{z_2,\dots,z_n=0}
\end{align*}
and $x_2 = x^2 e - \sum_{i=2}^n \xi_i v_i$. The main observation is that the marginals of a homogeneous Strongly-Rayleigh distribution can be computed from the derivatives of its generating polynomial 
\ifdefined\isarxiv 
(Fact~\ref{fac:marginal_polynomial_SR})
\fi .

Then, we can expand the term inside the expectation as 
\begin{align*}
(1-\frac{x}{2}\partial^2_{z_1})\Big( h(x_2+z_1v_1)g_2(x+z_1) \Big)\Big|_{z_1=0},
\end{align*}
using the fact that $\rank(v_1)_h\leq 1$ and the degree of $g_2(t)$ is at most 1.

Hence, we obtain our inductive step as
\begin{align*}
    & x^{d_{\mu}} \cdot 2^{-n} \cdot \E_{\xi \sim \mu} \left[ h(x e - \sum_{i=1}^n \xi_i v_i) \right] \\
    = ~& \frac{1}{2} ( 1 - \frac{x}{2} \partial_{z_1}^2 ) \left( \E_{\xi_2, \cdots, \xi_n} \Big[ h(xe - \sum_{i=2}^n \xi_i v_i + z_1 v_1) \cdot g_2(x+ z_1) \Big] \right) \Bigg|_{z_1=0}.
\end{align*}

\paragraph*{Applying the step inductively.}
Repeating the above process for $n$ times, we finally get
\begin{align*}
    x^{d_\mu}\cdot \E_{\xi\sim \mu}\left[h(x^2e - (\sum_{i=1}^n \xi_i v_i))\right]
    =  \sum_{T\subseteq [n]}(-\frac{x}{2})^{|T|}\partial_{z^{T}}^2\Big(h(x^2e+\sum_{i=1}^nz_iv_i)g_\mu(x{\bf 1}+z)\Big)\Bigg|_{z=0}.
\end{align*}
Then, we rewrite the partial derivatives as directional derivatives 
\ifdefined\isarxiv (see Definition \ref{def:directional-derivative} for detail) \fi. For any subset $T\subseteq [n]$ of size $k$, we have
\begin{align*}
    & ~ (-\frac{x}{2})^{k}\partial_{z^{T}}^2\Big(h(x^2e+\sum_{i=1}^nz_iv_i)g_\mu(x{\bf 1}+z)\Big)\Bigg|_{z=0} \\
    = &~ (-\frac{x}{2})^{k} \cdot 2^{k}\cdot \left(\prod_{i\in T}D_{v_i} \right)h(x^2e)\cdot g_\mu^{(T)}(x{\bf 1}),
\end{align*}
where $g_\mu^{(T)}(x{\bf 1})=\prod_{i\in T}\partial_{z_i} g_\mu(x{\bf 1}+z)\Big|_{z=0}$.
And by the homogeneity of $h$, it further equals to
\begin{align*}
    x^d \cdot (-\frac{1}{2})^{k}\partial_{z^T}^2\left(h(xe+\sum_{i=1}^n z_i v_i)g_\mu(x{\bf 1}+z)\right)\Bigg|_{z=0}.
\end{align*}

Therefore, we prove the following formula that relates the characteristic polynomial under SR distribution to the multivariate polynomial: 
\begin{align*}
    x^{d_\mu}\cdot \E_{\xi\sim \mu}\left[h(x^2e - (\sum_{i=1}^n \xi_i v_i))\right]
    =  x^d \cdot \prod_{i=1}^n (1-\frac{1}{2}\partial^2_{z_i})\Big(h(xe+\sum_{i=1}^nz_iv_i)g_\mu(x{\bf 1}+z)\Big)\Bigg|_{z=0} .
\end{align*}

\ifdefined\isarxiv
The complete proof can be found in Appendix \ref{sec:ag}.
\fi

\ifdefined\isarxiv
\paragraph*{Roadmap.}

We provide preliminary definitions and facts in Appendix \ref{sec:prelim}. In Appendix \ref{sec:kls}, we prove our first main result (Theorem \ref{thm:kls20_ours}), which is a hyperbolic generalization of Kadison-Singer result for a weaker condition (sum of squares of vectors is bounded).
In Appendix \ref{sec:ag}, we prove our second main result (Theorem \ref{thm:ag14_ours}), which is a hyperbolic extension of the Kadison-Singer result for strongly-Rayleigh distributions.
We put the sub-exponential algorithm for our main results in Appendix \ref{sec:sub-exp-alg}. In Appendix~\ref{sec:example}, we provide some examples of real-stable and hyperbolic polynomials. 
\fi

\addcontentsline{toc}{section}{References}
\ifdefined\isarxiv
\bibliographystyle{alpha}
\else 
\bibliographystyle{plainurl}
\fi
\bibliography{ref}

\ifdefined\isarxiv
\newpage
\appendix

\section{Preliminaries}\label{sec:prelim}
We gather several basic linear algebraic and analytic facts in the following subsections.
\ifdefined\isarxiv
\paragraph{Notations.}
\else
\paragraph*{Notations.}
\fi

For any positive integer $n$, we use $[n]$ to denote set $\{1,2,\cdots,n\}$. We use $\mathbf{1}$ to denote the all-one vector and $\mathbf{1}_i$ to denote the vector with one in the $i$-th coordinate and zero in other coordinates.

\subsection{Real-stable polynomials}

\begin{definition}
A multivariate polynomial $p\in \C[x_1,\cdots,x_m]$ is \emph{stable} if it has no zeros in the region $\{(x_1, \dots, x_m) : \Im(x_i) > 0 \text{ for all }  i \in [m]  \}$.  $p$ is \emph{real stable} if $p$ is stable and has real coefficients.  
\end{definition}

In the rest of this paper, we restrict our discussion into polynomials with real coefficients.

\begin{fact}\label{fac:uni_real_stable}
We say a univariate polynomial $p\in \R[t]$ is real-rooted iff it is real-stable.
\end{fact}

\begin{fact}[Equivalent definition of real-stable polynomial]\label{fac:equi_def_real_stable}
    A multivariate polynomial $p \in \R[x_1,\cdots,x_m]$ is real stable iff for any $a \in \R^m_{>0}$ and $b\in \R^m$, the univariate polynomial $p(at+ b)$ with respect to $t$ is not identically zero and is real-rooted.
\end{fact}

\begin{lemma}[Proposition 2.4, \cite{bb08}] \label{lem:determinantstable}
If $A_1, \dots, A_n$ are positive semidefinite symmetric matrices, then the polynomial
\begin{align*}
\det \left( \sum_{i=1}^n z_i A_i \right)
\end{align*}
is real stable.  
\end{lemma}

We also need that real stability is preserved under product (see \cite{bbl09}), restricting variables
to real values (see \cite[Lemma 2.4(d)]{wagner}), and taking $(1-\partial^2_{x_i})$ (see \cite[Corollary 2.8]{ag14}).
\begin{fact}[Closure operations of real-stable polynomials]\label{fac:closure_real_stable}
    Let $p,q \in \R[x_1,\dots,x_m]$ be two real stable polynomials. Then the following operations preserve real-stability:
    \begin{itemize}
        \item {\bf (Product)} $p\cdot q$.
        \item {\bf (Restriction to real values)} For any $a\in \R$, $p|_{x_1=a}=p(a,x_2,\ldots,x_m)\in\R[x_2,\ldots,x_m]$.
        \item {\bf (One minus second partial derivative)} For any $c\in \R_+,i\in [n]$, $(1 -c \cdot \partial_{x_i}^2) p(x_1, \dots, x_n)$.
    \end{itemize}
\end{fact}

\subsection{Hyperbolic polynomials}

\begin{definition}[Hyperbolic polynomials]\label{def:hyperbolic_polyn}
    A homogeneous polynomial $h\in \R[x_1,\cdots,x_m]$ is hyperbolic with respect to vector $e\in \R^m$ with $h(e) > 0$, if for all $x\in \R^m$, the univariate polynomial $t \mapsto h(te - x)$ only has real roots.
    
    Furthermore, if $h$ has degree $d$, fix any $x\in \R^m$ we can write
    \begin{align*}
        h(te - x) = h(e)\prod_{i=1}^d (t - \lambda_i(x)),
    \end{align*}
    where $\lambda_1(x) \ge \lambda_2(x) \ge \cdots \ge \lambda_d(x)$ are the real roots of the univariate polynomial $h(te - x)$. In particular,
    \begin{align*}
        h(x) = h(e) \prod_{i=1}^d \lambda_i(x).
    \end{align*}
    We denote  $\lambda_i(x)$ as the $i$-th eigenvalue of $x$.
\end{definition}

\begin{definition}[hyperbolicity cone]\label{def:hyperbolic_cone}
    Let $h \in \R[x_1,\cdots,x_m]$ be a degree $d$ hyperbolic polynomial with respect to  $e\in \R^m$.
    For any $x\in \R^m$, let $\lambda_1(x) \geq \cdots \geq \lambda_d(x) \in \R^d$ be the real roots of $h(te - x)$.
    Define the hyperbolicity cone of $h$ as 
    \begin{align*}
        \Gamma_{++}^h := \{x ~:~ \lambda_d(x) > 0\}.
    \end{align*}
    Furthermore, define the closure of $\Gamma_+^h$ as
    \begin{align*}
        \Gamma_{+}^h := \{x ~:~ \lambda_d(x) \ge 0\}.
    \end{align*}
\end{definition}

\begin{fact}\label{fac:real_stable_hyperbolic}
    For any $e \in \R^m_{>0}$ and
    any homogeneous real-stable polynomial $h \in \R[x_1,\cdots,x_m]$, we have $h$ is hyperbolic with respect to $e$. In other words, $\R_{>0}^m \subseteq \Gamma_+^h$.
\end{fact}
\begin{proof}
    For any $e \in \R^m_{>0}$ and for any $x \in \R^m$, by Fact \ref{fac:equi_def_real_stable} (with replacing $a$ by $e$ and $b$ by $x$), the uni-variate polynomial $h(te - x)$ is real rooted. By Definition \ref{def:hyperbolic_polyn} $h$ is hyperbolic with respect to direction $e$.
\end{proof}

\begin{definition}[Hyperbolic trace, rank, and spectral norm]\label{def:hyperbolic_trace}
    Let $h \in \R[x_1,\cdots,x_m]$ be a degree $d$ homogeneous hyperbolic polynomial with respect to  $e\in \R^m$. For any $x\in \R^m$, let $\lambda_1(x) \geq \cdots \geq \lambda_d(x) \in \R^d$ be the real roots of $h(te - x)$. Define
    \begin{align*}
        & \tr_h[x] := \sum_{i=1}^d \lambda_i(x), ~~~\rank(x)_h := |\{i ~:~ \lambda_i(x) \neq 0\}|, \\
        & \|x\|_h := \max_{i\in [d]} |\lambda_i(x)| = \max\{\lambda_1(x), -\lambda_d(x)\}.
    \end{align*}
\end{definition}

\begin{fact}[Hyperbolic norm in terms of largest root of characteristic polynomial]\label{fact:hyperbolic-norm-max-root}
    Let $h \in \R[x_1,\cdots,x_m]$ be a degree $d$ hyperbolic polynomial with respect to  $e\in \R^m$. For any $v_1,\cdots v_m \in \Gamma_+^h$ and any $s_1,\cdots,s_m \in \R$, we have
    \begin{align*}
        \left\| \sum_{i=1}^m s_i v_i \right\|_h = \lambda_{\max} \left( h\left(xe - \sum_{i=1}^n s_i v_i\right)\cdot h\left(xe + \sum_{i=1}^n s_i v_i\right) \right)
    \end{align*}
    where $\lambda_{\max}(f(x))$ is the maximum root of $f(x)$.
\end{fact}
\begin{proof}
For any vector $v\in \R^m$, it is easy to see that the $i$-th largest eigenvalue $\lambda_i(v)=-\lambda_{d-i+1}(-v)$. Then, we have
\begin{align*}
    \lambda_{\max} \left( h\left(xe - \sum_{i=1}^n s_i v_i\right)\cdot h\left(xe + \sum_{i=1}^n s_i v_i\right) \right) = &~ \max\left\{\lambda_1\left(\sum_{i=1}^m s_i v_i\right), \lambda_1\left(-\sum_{i=1}^m s_i v_i\right)\right\}\\
    = &~ \max\left\{\lambda_1\left(\sum_{i=1}^m s_i v_i\right), -\lambda_d\left(\sum_{i=1}^m s_i v_i\right)\right\}\\
    = &~ \left\| \sum_{i=1}^m s_i v_i \right\|_h.
\end{align*}
\end{proof}
\begin{remark}
    It is useful to think of hyperbolic polynomials as generalizations of the determinant polynomials. Let $X \in S^n(\R)$ be a symmetric matrix. Define $h: S^n(\R) \mapsto \R$ as
    \begin{align*}
        h(X) = \det(X).
    \end{align*}
      Then, $h$ is hyperbolic with respect to the identity matrix $I_n \in S^n(\R)$, since for all $X \in S^n(\R)$, the roots of $h(tI - X)$
    are the eigenvalues of $X$, thus $h(tI - X)$ is real-rooted.
    
    The basic concepts for hyperbolic polynomials in Definition \ref{def:hyperbolic_cone} and Definition \ref{def:hyperbolic_trace} also have analogues in linear algebra. To illustrate them, let $h$ be the determinant polynomial:
    \begin{itemize}
        \item The hyperbolicity cone of $h$ is 
        \begin{align*}
            \Gamma_+^h = \{X \in \R^{n \times n} ~:~ \lambda_n(X) > 0\} = \{X \in \R^{n \times n} ~:~ X \succ 0\}.
        \end{align*}
        \item For all $X\in S^n(\R)$, the hyperbolic trace of $X$ is
        \begin{align*}
            \tr_h[X] = \sum_{i=1}^n \lambda_i(X) = \tr[X].
        \end{align*}
        \item For all $X\in S^n(\R)$, the hyperbolic rank of $X$ is
        \begin{align*}
            \rank_h(X) = |\{i ~:~ \lambda_i(X) \neq 0\}| = \rank ( X ).
        \end{align*}
        \item For all $X\in S^n(\R)$, the hyperbolic spectral norm of $X$ is
        \begin{align*}
            \|X\|_h = \max_{i \in [n]} |\lambda_i(X)| = \|X\|.
        \end{align*}
        here $\|X\|$ denotes the spectral norm of $X$.
    \end{itemize}
\end{remark}

\begin{definition}[Directional derivative]
    Given a polynomial $h\in \R[x_1,\cdots,x_m]$ and a vector $v\in \R^m$, define the directional derivative of $h$ in direction $v$ as
    \begin{align*}
        D_v h := \sum_{i=1}^n v_i \frac{ \partial h }{ \partial x_i }.
    \end{align*}
\end{definition}

\begin{fact}[Equivalent definition of directional derivative]\label{def:directional-derivative}
    Given polynomial $h\in \R[x_1,\cdots,x_m]$ and vectors $x,v\in \R^m$,
    \begin{align}
    (D_v h) (x) = & ~ \frac{\d}{ \d t } h(x + t v), \label{eq:D_v_h} \\
    (D_v^2 h) (x) = & ~ \frac{\d^2}{\d t^2} h(x + t v). \label{eq:D_v^2_h}
    \end{align}
\end{fact}

\begin{fact}[Directional derivative of hyperbolic polynomials]\label{fac:direct_deriv_hyperbolic}
    Let $h\in \R[x_1,\cdots,x_m]$ denote a hyperbolic polynomial with respect to $e\in \R^m$. Let $v_1,\dots, v_n\in \R^m$ be $n$ vectors such that $\forall i\in [n]$, $\rank_h(v_i) \le 1$. Then, for any $t\in \R$,
    \begin{align*}
    (\prod_{i=1}^m D_{v_i} h) (te) = (\prod_{i=1}^n \partial_{z_i} h) (te + \sum_{i=1}^n z_i v_i).
    \end{align*}
    Moreover, for any $T \subseteq [m]$, we have
    \begin{align*}
    (\prod_{i \in T} D_{v_i} h) (te) = (\prod_{i\in T} \partial_{z_i} h) (te + \sum_{i=1}^n z_i v_i) \Big|_{z = 0}.
    \end{align*}
\end{fact}

\begin{fact}[First-order expansion of hyperbolic polynomial]\label{fac:first_order_expand}
Let $h \in \R[x_1,\cdots,x_m]$ be any hyperbolic polynomial. For any vectors $v_1,\dots,v_n\in \R^m$ such that $\forall i\in [n]$, $\rank_h(v_i)\leq 1$, and for any real vector $x\in \R^m$,
\begin{align*}
    h(x\pm\sum_{i=1}^n v_i)=\prod_{i=1}^n (1\pm\partial_{z_i}) h(x+\sum_{i=1}^n z_i v_i) \Bigg|_{z=0}.
\end{align*}
Furthermore, for any $S\subseteq [n]$,
\begin{align*}
    h(x\pm\sum_{i\in S} v_i)=\prod_{i\in S} (1\pm\partial_{z_i}) h(x+\sum_{i=1}^n z_i v_i) \Bigg|_{z=0}.
\end{align*}
\end{fact}

\begin{proof}
Prove by induction on $n$.

If $n=1$, we have
\begin{align*}
    h(x\pm v_1)=(1\pm D_{v_1})h(x)=(1\pm \partial_{z_1})h(x+z_1v_1)\Big|_{z_1=0},
\end{align*}
which follows from $\rank_h(v_1)\leq 1$.

Assume it holds for $n=k$.

When $n=k+1$, let $x'=x\pm \sum_{i=1}^{k} v_i$. We have
\begin{align*}
    h(x\pm \sum_{i=1}^{k+1} v_i)=&~ h(x'\pm v_{k+1})\\
    = &~ (1\pm \partial_{z_{k+1}})h(x'+z_{k+1}v_{k+1})\Big|_{z_{k+1}=0}\\
    = &~ (1\pm \partial_{z_{k+1}})h(x+z_{k+1}v_{k+1}\pm \sum_{i=1}^k v_i)\Big|_{z_{k+1}=0}\\
    = &~ (1\pm \partial_{z_{k+1}}) \prod_{i=1}^k (1\pm \partial_{z_i}) h(x+\sum_{i=1}^{k+1}z_iv_i)\Big|_{z=0}\\
    = &~ \prod_{i=1}^{k+1} (1\pm \partial_{z_i}) h(x+\sum_{i=1}^{k+1}z_iv_i)\Big|_{z=0},
\end{align*}
where the second step follows from $\rank_h(v_{k+1})\leq 1$, the forth step follows from the induction hypothesis, and the last step follows from the operators $(1\pm \partial_{z_i})$ and $(1\pm \partial_{z_j})$ commute for $i\ne j\in [k+1]$.

Hence, the fact is proved. And for the furthermore part, it follows from the remaining variables $z_i$ for $i\notin S$ will disappear when we set $z=0$.
\end{proof}

\subsection{Interlacing families}
We recall the definition and properties of interlacing families from \cite{mss15a}.
\begin{definition}[Interlacing polynomials and common interlacing] \label{def:interlacing}
We say a real rooted polynomial $g(x) = C \prod_{i=1}^{n-1} (x - \alpha_i)$ \emph{interlaces} the real rooted polynomial $f(x) = C' \prod_{i=1}^n (x - \beta_i)$ if 
\begin{align*}
\beta_1 \leq \alpha_1 \leq \dots \leq \alpha_{n-1} \leq \beta_n.
\end{align*}
We say the polynomials $f_1, \dots, f_k$ have a \emph{common interlacing} if there is a polynomial $g$ that interlaces each of the $f_i$.
\end{definition}

The following lemma relates the roots of a sum of polynomials to those of a common interlacing. 

\begin{lemma}[Lemma 4.2, \cite{mss15a}]\label{lem:interlacing_ub_root}
Let $f_1, \dots, f_k$ be degree $d$ real rooted polynomials with positive leading coefficients.  Define 
\begin{align*}
f_{\emptyset} := \sum_{i=1}^k f_i.
\end{align*}
If $f_1, \dots, f_k$ have a common interlacing, then there exists an $i \in [k]$ such that the largest root of $f_i$ is at most the largest root of $f_{\emptyset}$.
\end{lemma}

\begin{definition}[Definition 4.3, \cite{mss15a}] \label{def:interlacingfamily}
Let $S_1, \dots, S_n$ be finite sets.  
For each choice of assignment $(s_1, \dots, s_n) \in S_1 \times \dots \times S_n$, let $f_{s_1, \dots, s_n}(x)$ be a real rooted degree $d$ polynomial with positive leading coefficient.  For a partial assignment $s_1, \dots, s_k \in S_1 \times \dots \times S_k$ for $k < n$, we define
\begin{equation} \label{eq:conditionalpoly}
f_{s_1, \dots, s_k } := \sum_{s_{k+1} \in S_{k+1}, \dots, s_n \in S_n} f_{s_1, \dots, s_k, s_{k+1}, \dots, s_n}.
\end{equation}
Note that this is compatible with our definition of $f_{\emptyset}$ from Lemma \ref{lem:interlacing_ub_root}.  We say that the polynomials $\{f_{s_1, \dots, s_n} \}$ form an \emph{interlacing family} if for all $k = 0, \dots, n-1$ and all $(s_1, \dots, s_k) \in S_1 \times \dots \times S_k$, the polynomials 
\begin{align*}
\Big\{f_{s_1,\cdots,s_k, t}:~t\in S_{k+1}\Big\}  
\end{align*}
have a common interlacing. 
\end{definition}

The following lemma relates the roots of the interlacing family to those of $f_{\emptyset}$. 
\begin{lemma}[Theorem 4.4, \cite{mss15a}] \label{lem:rootbound}
Let $S_1, \dots, S_n$ be finite sets and let $\{f_{s_1, \dots, s_n}\}$ be an interlacing family. Then there exists some $(s_1, \dots, s_n) \in S_1 \times \dots \times S_n$ so that the largest root of $f_{s_1, \dots, s_n}$ is upper bounded by the largest root of $f_{\emptyset}$.
\end{lemma}

Finally, we recall a relationship between real-rootedness and common interlacings which has been discovered independently several times \cite{dg94, f80, cs07}.
\begin{lemma}[\cite{dg94, f80, cs07}] \label{lem:commoninterlacing}
Let $f_1, \dots, f_k$ be univariate polynomials of the same degree with positive leading coefficient.  Then $f_1, \dots, f_k$ have a common interlacing if and only if $\sum_{i=1}^k \alpha_i f_i$ is real rooted for all convex combinations $\alpha_i$, $\sum_{i=1}^k \alpha_i = 1$.
\end{lemma}

\subsection{Barrier method}\label{sec:prelim-barrier}

\begin{definition}[Upper barrier of the roots of a polynomial]
For a multivariate polynomial $p(z_1,\ldots,z_n)$,
we say $z\in\R^n$ is above all roots of $p$
if for all $t\in\R^n_+$,
\begin{align*} p(z+t)> 0. \end{align*}
We use $\Ab_p$ to denote the set of points which
are above all roots of $p$.
\end{definition}

We use the barrier function as in \cite{mss15b} and \cite{kls20}.
\begin{definition}[Barrier function]
    For a multivariate polynomial $p$ and $z \in \Ab_p$, the barrier function of $p$ in direction $i$ at $z$ is
    \begin{align*}
    \Phi_p^i : = \frac{\partial_{z_i} p( z ) }{p( z )}.
    \end{align*}
\end{definition}

We will make use of the following lemma that controls the deviation of the roots after applying a second order differential operator. This lemma is a slight variation of Lemma 5.3 in \cite{kls20}.

\begin{lemma}[Lemma 5.3 of \cite{kls20}]\label{lem:real_stable_cone}
Suppose that $p(z_1,\cdots,z_m)$ is real stable and $z \in {\bf Ab}_p$. For any $c \in [0,1]$ and $i\in [m]$, if 
\begin{align}
    \Phi_p^i(z) < \sqrt{1/c}, \label{eqn:barrier_cond_1}
\end{align}
then $z \in {\bf Ab}_{(1-c \cdot \partial_{z_i}^2) p}$.
If additionally for $\delta > 0$,
\begin{align}
    c \cdot \left( \frac{2}{\delta} \Phi_p^i(z) + ( \Phi_p^i(z) )^2 \right) \leq 1, \label{eqn:barrier_cond_2}
\end{align}
then, for all $j \in [m]$,
\begin{align*}
    \Phi^j_{(1 - c \cdot \partial_i^2) p } ( z + \delta {\bf 1}_i ) \leq \Phi_p^j(z).
\end{align*}
\end{lemma}
\begin{remark}
We choose $c=1/2$ when we use the above lemma in Section \ref{sec:kls_bound_root} and Section \ref{sec:ag_bound_root}.
\end{remark}

\begin{lemma}[Multi-dimensional convexity, \cite{tao13}]\label{lem:multi_phi_convex}
Let $p(z_1,\ldots,z_m)$ be a real stable polynomial of $m$ variables. For any $i\in [m]$,
\begin{align*}
    (-1)^k \frac{\partial^k}{\partial z_j^k} \Phi_p^i(x) \geq 0
\end{align*}
for all $k=0,1,2,\ldots$ and $x\in \Ab_p$. 
\end{lemma}
\section{Hyperbolic Extension of Kadison-Singer for Standard Deviations}\label{sec:kls}
The goal of this section is to prove Theorem~\ref{thm:kls20_ours_formal}:
\begin{theorem}[Formal statement of Theorem~\ref{thm:kls20_ours}]\label{thm:kls20_ours_formal}
Let $h\in \R[x_1,\dots,x_m]$ denote a hyperbolic polynomial with respect to a hyperbolic direction $e\in \Gamma_{++}^h$, where $\Gamma_{++}^h\subseteq \R^m$ is the hyperbolicity cone of $h$. Let $\xi_1,\cdots,\xi_n$ denote $n$ independent random variables with $\E[\xi_i]=\mu_i$ and $\Var[\xi]=\tau_i^2$. Let $v_1, \dots, v_n \in \Gamma_{+}^h$ be $n$ vectors such that $\forall i\in [n]$, $\rank_h(v_i) \leq 1$. Suppose 
\begin{align*}
    \sigma^2 := \Big\| \sum_{i=1}^n \tau_i^2\tr_h[v_i] v_i \Big\|_h.
\end{align*}
Then,
\begin{align*}
\Pr_{\xi_1,\cdots,\xi_n} \left[ \Big\| \sum_{i=1}^n (\xi_i-\mu_i) v_i \Big\|_h \leq 4 \sigma \right] >0.
\end{align*} 
\end{theorem}

We will introduce the preliminary facts in Section \ref{sec:kls_prelim}. 
In Section \ref{sec:kls_interlacing}, we define the family of hyperbolic characteristic polynomials and show that it forms an interlacing family. Therefore, it remains to upper-bound the largest root of the average of the interlacing polynomial family, i.e. the mixed characteristic polynomial. We remark that we require a lemma that will be proved later in Section \ref{sec:kls_expectation}.

In Section \ref{sec:kls_expectation}, we reduce the problem of upper-bounding the largest root of the mixed characteristic polynomial to an easier task of upper-bounding the largest root of a multivariate polynomial.  
In Section \ref{sec:kls_bound_root}, we upper bound the largest root of the multivariate polynomial using multivariate barrier method. 
Finally, in Section \ref{sec:kls-proof-main}, we prove Theorem \ref{thm:kls20_ours_formal}.

\subsection{Preliminaries}\label{sec:kls_prelim}

In this section, we state several useful facts about hyperbolic polynomials.
We first introduce some important properties of the derivatives of hyperbolic polynomial:

\begin{theorem}[Theorem 3.1 in \cite{b18} and known in \cite{g59,bgls01,r06}]\label{thm:hyperbolic_cone_derivative}
Let $h\in \R[x_1,\cdots,x_m]$ be a hyperbolic polynomial and let $v \in \Gamma_+$ be a vector such that $D_v h \not\equiv 0$. Then \\
1) $D_v h$ is hyperbolic with hyperbolicity cone containing $\Gamma_{++}$. \\
2) The polynomial $h(x) - y \cdot D_v h(x) \in \R[x,y]$ is hyperbolic with hyperbolicity cone containing $\Gamma_{++} \times \{ y : y \leq 0 \}$. Further, we have $( h ( x) + y \cdot D_v h ( x ) ) \cdot ( h(x) - y \cdot D_v h(x) ) \in \R[x,y] $ is hyperbolic with hyperbolicity cone containing $\Gamma_{++} \times \{ y : y \leq 0 \}$. \\
3) The rational function $x \rightarrow \frac{ h(x) }{ D_v h(x) }$ is concave on $\Gamma_{++}$. 
\end{theorem}

The following lemma correlates hyperbolic trace to directional derivative:

\begin{fact}[Correlation between hyperbolic trace and derivative]\label{fac:related_to_D_v_h}
Let $h\in \R[x_1,\cdots,x_m]$ denote a hyperbolic polynomial with respect to $e \in \R^m$. For any $v\in \R^m$ and $\alpha \in \R$, we have
\begin{align*}
    \tr_h[v] = \alpha \cdot \frac{D_v h(\alpha e)}{ h(\alpha e)}.
\end{align*}
\end{fact}
\begin{proof} By  Theorem \ref{thm:hyperbolic_cone_derivative}, $D_v h$ is hyperbolic, and thus is homogeneous. By Vieta's formula for the sum of roots of a polynomial, we have 
\begin{align*}
    \tr_h[ v ] =  \frac{ D_v h(e) }{ h(e) }
\end{align*}
It follows that $h$ and $D_v h$ are homogeneous, and the degree of $h$ is one larger than that of $D_v h$. Therefore,
\begin{align*}
    \alpha \cdot \frac{D_v h(\alpha e)}{ h(\alpha e)} 
    = \frac{D_v h(e)}{ h(e)} = \tr_h[v].    
\end{align*}
\end{proof}

\begin{fact}\label{fac:D_v_h_over_h_concave}
Let $h\in \R[x_1,\cdots,x_m]$ denote a hyperbolic polynomial with respect to the direction $e\in \Gamma_{++}$ and let $\alpha > 0$. 
For any $M,v \in \Gamma_{+}$, we have
\begin{align*}
    \frac{ D_{v} h(\alpha e + M) }{ h ( \alpha e + M ) } \leq \frac{ D_{v} h(\alpha e) }{ h(\alpha e) }.
\end{align*}
\end{fact}
\begin{proof}
It suffices to show
\begin{align*}
    \frac{ h(\alpha e + M) }{ D_{v} ( \alpha e + M ) } \geq  \frac{ h(\alpha e) }{ D_{v} h(\alpha e) }.
\end{align*}

By Theorem \ref{thm:hyperbolic_cone_derivative}, we know that rational function $x \mapsto \frac{h(x)}{ D_v h(x)}$ is concave on $\Gamma_{++}$. Then we know that
\begin{align*}
    \frac{ h(\alpha e + M) }{ D_{v} ( \alpha e + M ) } \geq & ~ \frac{1}{2}\frac{ h(2\alpha e) }{ D_{v} h(2\alpha e) } + \frac{1}{2}\frac{ h(2M) }{ D_{v} h(2M) }  \\
    = &~ \frac{ h(\alpha e) }{ D_{v} h(\alpha e) } + \frac{ h(M) }{ D_{v} h(M) }\\
    \geq & ~ \frac{ h(\alpha e) }{ D_{v} h(\alpha e) }.
\end{align*} 
where the last step follows from $\frac{ h(M) }{ D_{v} h(M) } \geq 0$ (By \cite{r06}).
\end{proof}

\subsection{Defining interlacing family of characteristic polynomials}\label{sec:kls_interlacing}

In this section, we consider the following interlacing family we will crucially use to prove Theorem \ref{thm:kls20_ours_formal}.

\begin{definition}[Interlacing Family of Theorem \ref{thm:kls20_ours_formal}]\label{def:interlacing-family-kls20}
    Let $h\in \R[x_1,\dots,x_m]$ denote a hyperbolic polynomial with respect to hyperbolic direction $e\in \Gamma_{++}^h$. Let $\xi_1,\dots,\xi_n$ denote $n$ independent random variables with finite supports and $\E[\xi_i]=\mu_i$ for $i\in [n]$. Let $v_1, \dots, v_n \in \Gamma_{+}^h$ be $n$ vectors such that $\mathrm{rank}_h(v_i) \leq 1$ for all $i \in [n]$. 
    For each $s = (s_1,\dots,s_n)$ where $s_i\in \supp(\xi_i)$, let $p_{s} \in \R [x]$ define the following polynomial:
    \begin{align*}
        p_{\bf s}(x) := \left(\prod_{i=1}^n p_{i, s_i}\right)\cdot h\left(x e + \sum_{i=1}^n (s_i-\mu_i)  v_i  \right) \cdot h\left(x e - \sum_{i=1}^n (s_i-\mu_i)  v_i  \right)
    \end{align*}
    where $p_{i, s_i} := \Pr_{\xi_i}[\xi_i = s_i]$.
    Let $\mathcal{P}$ denote the following family of polynomials:
    \begin{align*}
        \mathcal{P} := \left\{ p_{(s_1,\cdots, s_\ell)} = \sum_{\substack{t_{\ell+1} ,\cdots,t_{n}:\\t_j\in \supp(\xi_j)~\forall j\in \{\ell+1,\dots,n\}}} p_{(s_1,\cdots,s_\ell, t_{\ell+1},\cdots,t_n)} : ~ \ell \in [n], s_i \in \supp(\xi_i)~\forall i\in [\ell]\right\}.
    \end{align*}
\end{definition}

\begin{lemma}[Interlacing Family of Theorem \ref{thm:kls20_ours_formal}]\label{lem:interlacing-family-kls20}
    The polynomial family $\mathcal{P}$ defined in Definition \ref{def:interlacing-family-kls20} is an interlacing family.
\end{lemma}

\begin{proof}
    Fix any $\ell \in [n-1]$, and fix $s_1,\cdots,s_\ell$ as any  partial assignment of $\xi_1,\cdots,\xi_\ell$,  i.e. $s_i \in \{\pm 1\}$ for all $i\in \supp(\xi_i)$.
    
    It is easy to see that the polynomial $p_{(s_1,\dots,s_\ell)}$ can be written as:
    \begin{align*}
        \left( \prod_{i=1}^\ell p_{i, s_i} \right) \E_{{\xi}_{\ell+1}, \dots, \xi_{n}} \left[ h \left( x e + \left( \sum_{i=1}^\ell (s_i-\mu_i) v_i +\sum_{j=\ell+1}^n (\xi_j-\mu_j) v_j \right) \right)\cdot  \right.\\
        \left. h \left( x e - \left( \sum_{i=1}^\ell (s_i-\mu_i) v_i +\sum_{j=\ell+1}^n (\xi_j-\mu_j) v_j \right) \right)\right]
    \end{align*}
    
    Let $\{c_1,\dots,c_k\}$ be the support of $\xi_{\ell+1}$. Then, by Lemma \ref{lem:commoninterlacing}, it suffices to show that for any $\alpha \in \R_{\geq 0}^k$ with $\sum_{i=1}^k \alpha_i = 1$,
    the polynomial
    \begin{align*}
        \alpha_1 p_{(s_1,\cdots,s_\ell, c_1)} (x) +\cdots + \alpha_k p_{(s_1,\cdots,s_\ell, c_k)} (x)
    \end{align*}
    is real-rooted. We interpret $\alpha$ as the probability density of the $(\ell+1)$-th random variable, i.e. let $p_{\ell +1, c_t} = \alpha_t$ for all $t\in [k]$. Then we can define a new random variable $\wt{\xi}_{\ell+1}$ with the same support as $\xi_{\ell + 1}$ and $\Pr[\wt{\xi}_{\ell+1} = c_t] = p_{\ell +1, c_t}$ for all $t \in [k]$.
    
    Notice that
    \begin{align*}
        \wt{p_{\bf s}}(x) :=&~ \left( \prod_{i=1}^\ell p_{i, s_i} \right)\\
        &~~ \times\E_{\wt{\xi}_{\ell+1}, \xi_{\ell+2}\dots, \xi_{n}} \left[ h \left( x e - \left( \sum_{i=1}^\ell (s_i-\mu_i) v_i + (\wt{\xi}_{\ell+1} - \mu_{\ell+ 1})v_{\ell+1}+\sum_{j=\ell+2}^n (\xi_j-\mu_j) v_j \right) \right)\right. \\
        & ~ \left.\hspace{36mm} \cdot h \left( x e + \left( \sum_{i=1}^\ell (s_i-\mu_i) v_i + (\wt{\xi}_{\ell+1} - \mu_{\ell+ 1})v_{\ell+1} + \sum_{j=\ell+2}^n (\xi_j-\mu_j) v_j \right) \right) \right]  \\
        =& \sum_{t = 1}^k \left( \prod_{i=1}^\ell p_{i, s_i} \right) p_{\ell+1, c_t} \\
        &~~ \times \E_{\xi_{\ell+2}, \dots, \xi_{n}} \left [
         h \left( x e - \Bigg(\sum_{i=1}^\ell (s_i-\mu_i) v_i + c_t v_{\ell+1} + \sum_{j=\ell+2}^n (\xi_j-\mu_j) v_j \Bigg) \right)\right. \\
          & ~ \left.\hspace{25mm} \cdot 
         h \left( x e  + \Bigg(\sum_{i=1}^\ell (s_i-\mu_i) v_i + c_t v_{\ell+1} + \sum_{j=\ell+2}^n (\xi_j-\mu_j) v_j \Bigg) \right) \right] \\
        =& \sum_{i=1}^k \alpha_i p_{{\bf s}, c_i} (x). 
    \end{align*}  
    
    By Lemma \ref{lem:kls_expectation_hyperbolic} (we remark that it does not require any result in the current section. We will prove Lemma \ref{lem:kls_expectation_hyperbolic} in Section \ref{sec:kls_expectation}),
    $\wt{p_{\bf s}}$ is real-rooted, 
    which completes the proof that $\mathcal{P}$ is an interlacing family.
\end{proof}

Lemma \ref{lem:interlacing-family-kls20} implies the following corollary, which is a hyperbolic version of Proposition 4.1 in \cite{kls20}:

\begin{corollary}[Upper Bound of the Largest Root of Interlacing Family] \label{cor:kls-interlacing}
Let $h\in \R[x_1,\cdots,x_m]$ denote a hyperbolic polynomial  with corresponding hyperbolic direction $e\in \Gamma_{++}^h$. Let $\xi_1,\dots,\xi_n$ denote $n$ independent random variables with finite supports and $\E[\xi_i]=\mu_i$ for $i\in [n]$.
For any $v_1, \dots v_n \in \Gamma^h_{++}$ such that $\rank_h(v_i)\leq 1$ for all $i \in [n]$, 
there exists an sign assignment $s = (s_1, \dots, s_n)\in \supp(\xi_i)\times \cdots \supp(\xi_n)$,
such that
\begin{align*}
\left \| \sum_{i=1}^m (s_i-\mu_i) v_i  \right \|_h
\end{align*}
is at most the largest root of 
\begin{align*}
p_{\emptyset}(x) = \E_{ \xi_1, \dots, \xi_m } \left[ h \Big( x e +  \sum_{i=1}^m (\xi_i-\mu_i) v_i \Big) \cdot h \Big( x e - \sum_{i=1}^m (\xi_i-\mu_i) v_i \Big) \right].
\end{align*}
\end{corollary}

\begin{proof}

Let $\mathcal{P}$ be the interlacing family defined in Definition \ref{def:interlacing-family-kls20}. Then by Lemma \ref{lem:interlacing-family-kls20}, $\mathcal{P}$ is an interlacing family.

For any fixed $\ell \in [n]$ and the first $\ell$ assignments $(s_1,\cdots,s_\ell)$ such that $s_i\in \supp(\xi_i)$ for $i\in [\ell]$, notice that 
\begin{align*}
    p_{s_1,\cdots,s_\ell}(x) = \left( \prod_{i=1}^\ell p_{i, s_i} \right) \E_{\xi_{\ell+1}, \dots, \xi_{n}} \Bigg[ & ~ h \left( x e - \left( \sum_{i=1}^\ell (s_i-\mu_i) v_i + \sum_{j=\ell+1}^n  (\xi_j-\mu_j) v_j \right) \right)\cdot  \\
    & ~  h \left( x e + \left( \sum_{i=1}^\ell (s_i-\mu_i) v_i + \sum_{j=\ell+1}^n (\xi_j-\mu_j) v_j \right) \right) \Bigg] ,
\end{align*}
and
\begin{align*}
    p_{\emptyset} = \E_{ \xi_1, \dots, \xi_m } \left[ h \Big( x e +  \sum_{i=1}^m (\xi_i-\mu_i) v_i \Big) \cdot h \Big( x e - \sum_{i=1}^m (\xi_i-\mu_i) v_i \Big) \right].
\end{align*}

Therefore, by Lemma \ref{lem:rootbound}, there exists a sign assignment
  \begin{align*}
      (s_1,\dots,s_n) \in \supp(\xi_i)\times \cdots \supp(\xi_n)
  \end{align*}
  such that the largest root of $p_s$ is upper bounded by the largest root of
  \begin{align*}
      \E_{ \xi_1, \dots, \xi_m } \left[ h \Big( x e +  \sum_{i=1}^m (\xi_i-\mu_i) v_i \Big) \cdot h \Big( x e - \sum_{i=1}^m (\xi_i-\mu_i) v_i \Big) \right].
  \end{align*}
Then by Fact \ref{fact:hyperbolic-norm-max-root}, we have $\left \| \sum_{i=1}^m (s_i-\mu_i) v_i  \right \|_h=\lambda_{\max}(p_{\bf s})$, which is upper bounded by the maximum root of $p_{\emptyset}$.

\end{proof}

\subsection{From mixed characteristic polynomial to multivariate polynomial}\label{sec:kls_expectation}

In this section, we want to show that the mixed characteristic polynomial, i.e. the average of the interlacing family defined in Definition \ref{def:interlacing-family-kls20}:
\begin{align*}
    p_{\emptyset} = \E_{ \xi_1, \dots, \xi_m } \left[ h \Big( x e +  \sum_{i=1}^m (\xi_i-\mu_i) v_i \Big) \cdot h \Big( x e - \sum_{i=1}^m (\xi_i-\mu_i) v_i \Big) \right]
\end{align*}
equals to the following multivariate polynomial after taking $z_1=\cdots = z_n = 0$. 
\begin{align*}
    \prod_{i=1}^n \Big(1- \frac{1}{2} \frac{\partial^2}{ \partial z_i^2 }\Big) h \Big( xe - Q + \sum_{i=1}^n z_i\tau_i v_i \Big) \cdot h \Big( xe + Q + \sum_{i=1}^n z_i\tau_i v_i \Big) \in \R[x,z_1,\cdots,z_n]
\end{align*}
The largest root of this multivariate polynomial is relatively easy to upper-bound using barrier argument. We will describe the details in Section \ref{sec:kls_bound_root}.

The main lemma of this section is as follows:

\begin{lemma}[Hyperbolic version of Proposition 3.3 in \cite{kls20}]\label{lem:kls_expectation_hyperbolic}
Let $h\in \R[x_1,\cdots,x_m]$ denote a hyperbolic polynomial with respect to a hyperbolic direction $e\in \Gamma_{++}^h$.
Let $v_1, \dots v_n \in \Gamma_+^h$ such that $\forall i \in [n]$, $\rank_h(v_i) \le 1$. Let $\xi_1,\dots,\xi_n$ denote $n$ independent random variables such that $\E[\xi_i] = \mu_i$ and $\Var[\xi_i]=\tau_i^2$. For any $Q \in \R^m$, we have
\begin{align}
& ~ \E_{\xi_1,\dots, \xi_n} \left[ h \Big(x e - (Q + \sum_{i=1}^n (\xi_i-\mu_i) v_i ) \Big) h \Big(x e + (Q + \sum_{i=1}^n (\xi_i-\mu_i) v_i ) \Big) \right] \nonumber \\
 = & ~ 
  \prod_{i=1}^n \Big(1- \frac{1}{2} \frac{\partial^2}{ \partial z_i^2 }\Big) \Big|_{z_i = 0} h \Big( xe - Q + \sum_{i=1}^n z_i\tau_i v_i \Big) \cdot h \Big( xe + Q + \sum_{i=1}^n z_i\tau_i v_i \Big). \label{eqn:kls_expectation_mix_hyperbolic}
\end{align}
Moreover, this is a real-rooted polynomial in $x$.
\end{lemma}

\begin{proof}%
We first show Eqn.~\eqref{eqn:kls_expectation_mix_hyperbolic} by induction.
Our induction hypothesis will be that for any $0 \leq k \leq n$,
\begin{align}
& ~ \E_{\xi_1,\dots, \xi_n} \left[ h \Big(x e - (Q + \sum_{i=1}^n (\xi_i-\mu_i) v_i ) \Big) h \Big(x e + (Q + \sum_{i=1}^n (\xi_i-\mu_i) v_i ) \Big) \right] \notag\\
= & ~ \E_{\xi_{k+1}, \dots, \xi_n}
    \prod_{i=1}^k \Big(1- \frac{1}{2} \frac{\partial^2}{ \partial z_i^2 }  \Big) \Big|_{z_i = 0} h \Big( xe - Q - \sum_{i=k+1}^n (\xi_i-\mu_i) v_i + \sum_{j=1}^k z_j \tau_j v_j \Big) \notag\\
    &~\times h \Big( xe + Q + \sum_{i=k+1}^n (\xi_i-\mu_i) v_i + \sum_{j=1}^k z_j\tau_j v_j \Big) \label{eqn:kls_expection_ih}
\end{align}

The base case, $k = 0$ trivially holds as we get the same formula on both sides.

For the inductive step, suppose the induction hypothesis holds for any $k \le \ell$ where $0\le \ell < n$.
Applying Claim~\ref{clm:hyperbolic_lemma_3.1_in_kls19} to the right-hand-side of Eqn. \eqref{eqn:kls_expection_ih} when letting $k = \ell$ yields
\begin{align}
    & ~ \E_{\xi_1,\dots, \xi_m} \left[ h \Big(x e - (Q + \sum_{i=1}^n (\xi_i-\mu_i) v_i ) \Big) h \Big(x e + (Q + \sum_{i=1}^n (\xi_i-\mu_i) v_i ) \Big) \right]  \nonumber \\
    = & ~ \E_{\xi_{\ell+2}, \dots, \xi_m}
        \prod_{i=1}^{\ell+1} \Big(1- \frac{1}{2} \frac{\partial^2}{ \partial z_i^2 }\Big) \Big|_{z_i = 0} h \Big( xe - Q - \sum_{i=\ell+2}^n (\xi_i-\mu_i) v_i + \sum_{j=1}^{\ell+1} z_j\tau_j v_j \Big) \nonumber \\
        &~ \times h \Big( xe + Q + \sum_{i=\ell+2}^n (\xi_i-\mu_i) v_i + \sum_{j=1}^{\ell+1} z_j\tau_j v_j \Big) \label{eqn:kls_expectation_rs}
\end{align}
This completes the proof of Eqn. \eqref{eqn:kls_expectation_mix_hyperbolic}.

We now show that  Eqn. (\ref{eqn:kls_expectation_mix_hyperbolic}) is real-rooted in $x$. 
By Claim~\ref{clm:hyper_linear_restriction}, we know that
\begin{align*}
    h \Big( xe - Q + \sum_{i=1}^n z_i\tau_i v_i \Big)~~\text{and}~~h \Big( xe + Q + \sum_{i=1}^n z_i\tau_i v_i \Big)\in \R[x,z_1,\dots,z_n],
\end{align*}
are real-stable polynomials.
Then, by Fact~\ref{fac:closure_real_stable}, we get that
\begin{align*}
    h \Big( xe - Q + \sum_{i=1}^n z_i\tau_i v_i \Big) \cdot h \Big( xe + Q + \sum_{i=1}^n z_i\tau_i v_i \Big)
\end{align*}
is also real-stable. And by Fact~\ref{fac:closure_real_stable} again, the operator
\begin{align*}
    \prod_{i=1}^n \Big(1- \frac{1}{2} \frac{\partial^2}{ \partial z_i^2 }\Big) \Big|_{z_i = 0}
\end{align*}
preserves the real-stability. Hence, by Fact~\ref{fac:uni_real_stable}, the RHS of Eqn.~\eqref{eqn:kls_expectation_mix_hyperbolic} is real-rooted.
\end{proof}

The following statements are crucially used in the proof of Lemma \ref{lem:kls_expectation_hyperbolic}:

\begin{claim}[Hyperbolic version of Lemma 3.1 in \cite{kls20}]\label{clm:hyperbolic_lemma_3.1_in_kls19}
Let $h$ denote a hyperbolic polynomial. Let $\xi$ denote a random variables with $\E[\xi]=0$ and $\Var[\xi]=\tau^2$. For any $v \in \R^m$ such that $\rank_h(v) \le 1$, and any $x_1,x_2 \in \R^m$, we have
\begin{align*}
    \E_{\xi} [ h(x_1 - \xi v) \cdot h(x_2 + \xi v) ] = \left(1 - \frac{1}{2}\frac{\d ^2}{\d t^2}\right)\Bigg|_{t=0} h(x_1+t\tau v) h(x_2+t\tau v).
\end{align*}
\end{claim}

\begin{remark}
Claim~\ref{clm:hyperbolic_lemma_3.1_in_kls19} can be easily generalized for non-centered random variable $\xi$ with $\E[\xi]=\mu$:
\begin{align*}
    \E_{\xi} [ h(x_1 - (\xi-\mu) v) \cdot h(x_2 + (\xi-\mu) v) ] = \left(1 - \frac{1}{2}\frac{\d ^2}{\d t^2}\right)\Bigg|_{t=0} h(x_1+t\tau v) h(x_2+t\tau v).
\end{align*}
\end{remark}

\begin{proof}
Since $\rank_h(v) \le 1$, for all $k \geq 2$, $D_{v}^k h \equiv 0$.
Thus, for all $x_1 \in \R^m$, 
\begin{align*}
    h(x_1 - \xi v) = \left( \sum_{k=0}^{\infty} \frac{ (-\xi)^k D_v^k }{ k! } \right) h(x_1) = (1 - \xi D_v) h(x_1).    
\end{align*}
where the first step follows from Taylor expansion of $h(x_1 - \xi v)$ on $x_1$.

Similarly, for all $x_2 \in \R^m$, 
\begin{align*}
    h(x_2 + \xi v) = (1 + \xi D_v) h(x_2).    
\end{align*}

Therefore, 
\begin{align*}
    h(x_1 - \xi v) \cdot h(x_2 + \xi v) &= (1 - \xi D_v) h(x_1) \cdot (1 + \xi D_v) h(x_2) \\
    &= h(x_1)h(x_2)-\xi h(x_2)D_vh(x_1) + \xi h(x_1)D_vh(x_2)-\xi^2 D_vh(x_1)D_vh(x_2)
\end{align*}

Since $\E[\xi] = 0$ and $\Var[\xi]=\tau^2$, we have
\begin{align*}
    \E_{\xi} [ h(x_1 - \xi v) \cdot h(x_2 + \xi v) ] = &~ \left(1 - \frac{\E[\xi^2]}{2} D_{v}^2\right) h(x_1) h(x_2) + \E_{\xi}[\xi]\cdot (h(x_1)D_v h(x_2) - h(x_2)D_v h(x_1)) \\
    = &~ \left(1 - \frac{\tau^2}{2} D_{v}^2\right) h(x_1) h(x_2)\\
    = &~ \left(1 - \frac{1}{2}\frac{\d ^2}{\d t^2}\right)\Bigg|_{t=0} h(x_1+t\tau v) h(x_2+t\tau v).
\end{align*}
\end{proof}

\begin{claim}[Linear restriction of hyperbolic polynomial is real-stable]\label{clm:hyper_linear_restriction}
Let $h\in \R[x_1,\dots,x_m]$ be a hyperbolic polynomial with respect to $e\in \R^m$. Let $v_1,\dots, v_n \in \Gamma_+$ and $Q\in \R^m$. Define
\begin{align*}
    p(x, z) := h \Big( xe + \sum_{i=1}^n z_iv_i -Q \Big)\in \R[x,z_1,\dots,z_n].
\end{align*}
Then, $p(x,z)$ is a real-stable polynomial.    
\end{claim}
\begin{proof}
For any $a\in \R_{>0}^{n+1}$, $b\in \R^{n+1}$, we have
\begin{align*}
    p(at + b) = &~ h \Big( (a_1 t + b_1)e - Q + \sum_{i=1}^n (a_{i+1}t+b_{i+1})v_i \Big)\\
    = &~ h \Big( (a_1 e + \sum_{i=1}^n a_{i+1}v_i)t + b_1e - Q + \sum_{i=1}^n b_{i+1}v_i \Big).
\end{align*}

Since $e \in \Gamma_{++}$, $v_1,\dots, v_n\in \Gamma_+$ and $a_i>0$ for all $i\in [n+1]$, we have
\begin{align*}
    e':=a_1 e + \sum_{i=1}^n a_{i+1}v_i \in \Gamma_{++},
\end{align*}
which follows from $\Gamma_{++}$ is a cone. 

Since every vector in $\Gamma_{++}$ is a hyperbolic direction of $h$ (see e.g., \cite[Theorem 1.2, item 4]{b18}), we know that
$h$ is also hyperbolic with respect to the direction $e'$, which implies that $p(at+b)$ is real-rooted and not identical to the zero polynomial. 

Hence, by Fact~\ref{fac:equi_def_real_stable}, $p(x,z)$ is a real-stable polynomial.    
\end{proof}

\subsection{Applying barrier argument to bound the largest root of multivariate polynomial}\label{sec:kls_bound_root}

In this section, we upper bound the largest root of the following multivariate polynomial:
\begin{align*}
    \prod_{i=1}^n \Big(1- \frac{1}{2} \frac{\partial^2}{ \partial z_i^2 }\Big) h \Big( xe - Q + \sum_{i=1}^n z_i\tau_i v_i \Big) \cdot h \Big( xe + Q + \sum_{i=1}^n z_i\tau_i v_i \Big)
\end{align*}
using the real-stable version of the barrier method in \cite{kls20}.

\begin{definition}
Let $h\in \R[x_1,\cdots,x_m]$ be a hyperbolic polynomial of degree $d$ with respect to hyperbolic direction $e\in \R^m$. For any vectors $u,v\in \R^m$, we say $u\preceq v$ if
\begin{align*}
    \lambda_i(u) \leq \lambda_i(v) \quad \forall i\in [d],
\end{align*}
where $\lambda(u),\lambda(v)$ are the ordered eigenvalues of $u$ and $v$, respectively. 
\end{definition}

\begin{claim}\label{clm:prec_e}
Let $h$ be a hyperbolic polynomial of degree $d$ with respect to hyperbolic direction $e\in \R^m$. Let $u\in \R^m$ be any vector such that $u\preceq e$. Then we have $e-u\in \Gamma^h_+$.
\end{claim}
\begin{proof}
For any $i\in [d]$, we have
\begin{align*}
    \lambda_i(e-u) = 1-\lambda_{d-i}(u)\geq 0,
\end{align*}
where the last step follows from $\lambda_i(u)\leq \lambda_i(e)\leq 1$ for all $i\in [d]$.

Hence, $e-u\in \Gamma^h_+$.
\end{proof}

The goal of this section is to prove the following lemma:

\begin{lemma}\label{lem:kls_bound_root}
Let $h\in \R[x_1,\cdots,x_m]$ denote a hyperbolic polynomial with corresponding hyperbolic direction $e\in \Gamma^h_+$. Let $\xi_1,\dots,\xi_n$ denote $n$ independent random variables with finite supports and $\E[\xi_i]=\mu_i$ and $\Var[\xi_i]=\tau_i^2$ for $i\in [n]$. Let $v_1, \dots, v_n \in \Gamma^h_+$ 
such that $\sum_{i=1}^n \tau_i^2 \tr_h[v_i]v_i \preceq e$ and $\forall i\in [n]$, $\rank_h(v_i) \leq 1$. 

Then all the roots of the following $(n+1)$-variate polynomial
\begin{align*}
    \prod_{i=1}^n \Big(1- \frac{1}{2} \frac{\partial^2}{ \partial z_i^2 }\Big) \left(h \Big( xe + \sum_{i=1}^n z_i\tau_i v_i \Big)\right)^2 \in \R[x,z_1,\dots, z_n]
\end{align*}
lie below $(4,0,\cdots,0)\in \R^{n+1}$.
\end{lemma}

\begin{proof}
Define $(n+1)$-variate polynomial $P(x,z)\in \R[x,z_1,\dots, z_n]$ as
\begin{align*}
P ( x , z ) := \left( h \Big( x e + \sum_{i=1}^n z_i \tau_iv_i \Big) \right)^2.
\end{align*}
By Claim~\ref{clm:hyper_linear_restriction} and Fact~\ref{fac:closure_real_stable}, we can show that $P(x,z)$ is real-stable. Thus, we can apply the multivariate barrier method in \cite{kls20} with the barrier functions $\Phi_{P}^i(x,z)=\frac{\partial_{z_i}P(x,z)}{P(x,z)}$ for $i\in [n]$.

For $t > 0$, let $\delta_i = t \tau_i\tr_h[v_i]$ and let 
\begin{align*}
    \mathbf{\delta} = (\delta_1,\dots,\delta_n).    
\end{align*}
For some $\alpha(t) > t$ where $\alpha(t)$ is a parameter to be chosen later, we evaluate $P$ at
\begin{align*}
    (\alpha, -\delta) = (\alpha(t),-\delta_1,\dots,-\delta_n)
\end{align*}
to find that
\begin{align*}
P(\alpha(t), - \delta_1, \dots, - \delta_n)
= & ~  \left( h \Big( \alpha(t)  e - \sum_{i=1}^n \delta_i \tau_iv_i \Big) \right)^2 \\
= & ~  \left( h \Big( \alpha(t) e - t \sum_{i=1}^n \tau_i^2 \tr_h[v_i] v_i \Big) \right)^2\\
= &~ \left( h(e)\prod_{j=1}^d \Big(\alpha(t)-t\lambda_j\Big(\sum_{i=1}^n \tau_i^2\tr_h[v_i] v_i\Big)\Big) \right)^2
\end{align*}
where $d$ is the degree of $h$. Here the last step follows from the hyperbolicity of $h$, and the fact that the set of roots of 
\begin{align*}
    h \Big( \alpha(t) e - t \sum_{i=1}^n \tau_i^2 \tr_h[v_i] v_i \Big)
\end{align*}
are
\begin{align*}
    \left\{ \left( \frac{1}{\alpha(t)} \cdot \lambda_j \Big( \sum_{i=1}^n \tau_i^2\tr_h[v_i] v_i \Big) \right)^{-1} \right\}_{j \in [d]}
\end{align*}
Since $\sum_{i=1}^n \tau_i^2\tr_h[v_i]v_i \preceq e$, we have for all $j\in [d]$, $\lambda_j(\sum_{i=1}^n \tau_i^2\tr_h[v_i]v_i)\leq \lambda_j(e)=1$. 

Hence, by the assumption of $t< \alpha(t)$, we get that
\begin{align*}
    P(\alpha(t), - \delta_1, \dots, - \delta_n) \geq h(e)^2 (\alpha(t)-t)^{2d}>0.
\end{align*}
This implies that $(\alpha,-\mathbf{\delta}) \in \R^{m+1}$ is above the roots of $P(x,z)$, i.e. $(\alpha,-\mathbf{\delta}) \in \Ab_{P}$. 

Moreover, we can upper bound $\Phi^i_P (\alpha,-\mathbf{\delta})$ as follows:

\begin{align*}
\Phi^i_P ( \alpha(t), - \delta )
= & ~ \frac{ \partial_{z_i} P }{ P } \Big|_{ x = \alpha(t), z = -\mathbf{\delta} } \\
= & ~ \frac{2 h(xe + \sum_{i=1}^n z_i\tau_i v_i ) \partial_{z_i} h( xe + \sum_{i=1}^n z_i \tau_iv_i )} {h( x e + \sum_{i=1}^n z_i \tau_iv_i )^2} \Big|_{ x = \alpha(t), z = - \delta } \\
= & ~ 2 \cdot \frac{  \partial_{z_i} h( xe + \sum_{i=1}^n z_i\tau_i v_i ) }{ h( x e + \sum_{i=1}^n z_i\tau_i v_i ) } \Big|_{x = \alpha(t), z = - \delta} \\
= &~ 2\cdot \frac{(D_{\tau_iv_i}h)(\alpha e+\sum_{j=1}^n \delta_j \tau_jv_j)}{h(\alpha e + \sum_{j=1}^n \delta_j \tau_jv_j)}\\
= &~ 2\cdot \frac{(D_{\tau_iv_i}h)(\alpha e-t\sum_{j=1}^n \tau_j^2\tr_h[v_j] v_j)}{h(\alpha e -t \sum_{j=1}^n \tau_j^2\tr_h[v_j] v_j)}\\
= &~ 2\cdot \frac{(D_{\tau_iv_i}h)(\alpha e-te + t(e - \sum_{j=1}^n \tau_j^2\tr_h[v_j] v_j))}{h(\alpha e -te + t(e - \sum_{j=1}^n \tau_j^2\tr_h[v_j] v_j))}\\
\leq & ~  \frac{ 2(D_{\tau_iv_i} h)( \alpha e - t e ) }{ h(\alpha e - t e) } \\
\leq & ~ \frac{ 2\tr_h[\tau_iv_i] }{ \alpha - t }.
\end{align*}
where the second last step follows from $\sum_{i=1}^n \tau_j^2\tr_h[v_i] v_i \preceq e$, Claim~\ref{clm:prec_e} and Fact~\ref{fac:D_v_h_over_h_concave}. The last step follows from $ \frac{ (D_v h)(\beta e) }{ h(\beta e) } = \frac{\tr_h[v]}{\beta}$ by Fact~\ref{fac:related_to_D_v_h}. %

Since $\rank_h(v_i)\leq 1$, we have $\tr_h[v_i] = \|v_i\|_h$ for all $i\in [n]$. Since $\sum_{j\ne i}\tau_j^2\tr_h[v_j]v_j\in \Gamma^h_+$, by the monotonicity of the hyperbolic norm (Theorem 2.15 in \cite{hlj09}), we have
\begin{align*}
    (\tau_i\tr_h[v_i])^2 = \|\tau_i^2\tr_h[v_i]v_i\|_h\leq \left\|\sum_{j=1}^n \tau_j^2 \tr_h[v_j]v_j\right\|_h\leq \|e\|_h= 1
\end{align*}
for all $i\in [n]$.

Thus, we have
\begin{align*}
\max_{i \in [n]} ~~( \tau_i\tr_h[v_i] )^2 \leq 1.
\end{align*}

Choosing $\alpha(t) = 2t = 4$. We get
\begin{align}
    \Phi_P^i (\alpha, - \delta) \leq \frac{ 2\tr_h[ \tau_iv_i ] }{ \alpha - t}=\frac{ 2\tau_i\tr_h[ v_i ] }{ \alpha - t} = \frac{2\tau_i\tr_h[v_i]}{2}\leq 1 < \sqrt{2}. \label{eqn:kls_root_cond_1}
\end{align}
This coincides with Eqn. \eqref{eqn:barrier_cond_1} of Lemma \ref{lem:real_stable_cone}. Also from Fact \ref{fac:closure_real_stable} we know that $(1 - \frac{1}{2} \partial_{z_i}^2 )P$ is real stable. 

Thus by Lemma \ref{lem:real_stable_cone}, for all $i\in [n]$,
\begin{align*}
    (4,-\mathbf{\delta})\in \Ab_{(1 - \frac{1}{2} \partial_{z_i}^2)P }.    
\end{align*}

In addition, $\forall i\in [n]$, since $\delta_i  = t \tau_i\tr_h[v_i] = 2 \tau_i\tr_h[v_i]> 0$,
\begin{align}\label{eq:phi_less_than_one}
    \frac{1}{ \delta_i } \Phi_P^i(4,-\mathbf{\delta}) + \frac{1}{2} \Phi_P^i(4,-\mathbf{\delta})^2 
    \leq & ~ \frac{1}{2 \tau_i\tr_h[v_i]} \tau_i\tr_h[v_i] + \frac{1}{2} ( \tau_i \tr_h[v_i] )^2 \nonumber \\
    \leq & ~ \frac{1}{2} + \frac{1}{2} =  1.
\end{align}

This coincides with Eqn. \eqref{eqn:barrier_cond_2} of Lemma \ref{lem:real_stable_cone}. Therefore for all $j\in [n]$,
\begin{align}
    \Phi^j_{(1 - \frac{1}{2} \partial_{z_i}^2)P} (4, -\mathbf{\delta} + \delta_i {\bf 1}_i) \le \Phi_P^j(4, -\mathbf{\delta}). \label{eqn:kls_root_cond_2}
\end{align}

In particular, we have
\begin{align*}
    \Phi^2_{(1 - \frac{1}{2} \partial_{z_1}^2)P} (4, -\mathbf{\delta} + \delta_1 {\bf 1}_1) \le \Phi_P^2(4, -\mathbf{\delta})<\sqrt{2}.
\end{align*}

We also have $(4, -\delta + \delta_1{\bf 1}_1)\in \Ab_{(1-\frac{1}{2}\partial_{z_1}^2)P}$, which follows from $(4,-\delta)\in \Ab_{(1-\frac{1}{2}\partial_{z_1}^2)P}$. By Lemma~\ref{lem:multi_phi_convex} with $k=0$, we get that
\begin{align*}
    \Phi^2_{(1 - \frac{1}{2} \partial_{z_1}^2)P} (4, -\mathbf{\delta} + \delta_1 {\bf 1}_1)\geq 0.
\end{align*}

Hence, we have
\begin{align*}
    &\frac{1}{ \delta_2 } \Phi_{(1 - \frac{1}{2} \partial_{z_1}^2)P}^2(4,-\mathbf{\delta} + \delta_1 {\bf 1}_1) + \frac{1}{2} \Phi_{(1 - \frac{1}{2} \partial_{z_1}^2)P}^2(4,-\mathbf{\delta} + \delta_1 {\bf 1}_1)^2\\ \leq &~ \frac{1}{ \delta_2 } \Phi_P^2(4,-\mathbf{\delta}) + \frac{1}{2} \Phi_P^2(4,-\mathbf{\delta})^2\\
    \leq &~ 1,
\end{align*}
where the last step follows from Eqn. \eqref{eq:phi_less_than_one}.

Therefore, by Lemma~\ref{lem:real_stable_cone} again, we have
\begin{align*}
    \Phi^i_{(1 - \frac{1}{2} \partial_{z_2}^2)(1 - \frac{1}{2} \partial_{z_1}^2)P} (4, -\mathbf{\delta} + \delta_1 {\bf 1}_1+\delta_2{\bf 1}_2) \le \Phi^i_{(1 - \frac{1}{2} \partial_{z_1}^2)P} (4, -\mathbf{\delta} + \delta_1 {\bf 1}_1).
\end{align*}

Repeating this argument for each $i\in [n]$ demonstrates that 
\begin{align*}
    (4, -\mathbf{\delta} + \sum_{i=1}^n \delta_i {\bf 1}_i) &= (4, 0,0,\dots,0) \\
    &\in \Ab_{\prod_{i=1}^n (1 - \partial_{z_i}^2 / 2) P}
\end{align*}
i.e. $(4, 0,0,\dots,0)$ lies above the roots of 
\begin{align*}
    \prod_{i=1}^n \Big( 1 - \frac{1}{2} \frac{ \partial^2 }{ \partial z_i^2 } \Big)  \left( h \Big( x e + \sum_{i=1}^n z_i\tau_i v_i \Big) \right)^2.
\end{align*}

\end{proof}

\subsection{Combining together: proof of Theorem \ref{thm:kls20_ours_formal}}\label{sec:kls-proof-main}
In this section, we will combine the results from the previous section and prove Theorem \ref{thm:kls20_ours_formal}:

\begin{proof}[Proof of Theorem \ref{thm:kls20_ours_formal}]
Define $u_i := \frac{v_i}{\sigma}$. Note that $\sigma>0$ since $v_1,\dots,v_n$ are in the hyperbolicity cone of $h$. 

Then, we have
\begin{align*}
    \left\|\sum_{i=1}^n \tau_i^2\tr_h[u_i]u_i\right\|_h =  \left\|\sum_{i=1}^n \frac{\tau_i^2\tr_h[v_i]v_i}{\sigma^2}\right\|_h=1,
\end{align*}
where the first step follows from the linearity of the hyperbolic trace $\tr_h$, and the second step follows from $\|\cdot\|_h$ is a norm.
 
By Lemma~\ref{lem:kls_bound_root} and restricting to $z_i = 0$ for all $i\in [n]$, we have that $4$ lies above the largest root of the univariate polynomial
\begin{align*}
    \prod_{i=1}^n \Big(1- \frac{1}{2} \frac{\partial^2}{ \partial z_i^2 }\Big) \Big|_{z_i = 0} \left(h \Big( xe + \sum_{i=1}^n z_i\tau_i u_i \Big)\right)^2
\end{align*}
We then conclude by Lemma \ref{lem:kls_expectation_hyperbolic} that $4$ upper bounds the largest root of 
\begin{align*}
    p_{\emptyset} = \E_{\xi_1, \dots, \xi_n} \left[ h \Big( x e + \sum_{i=1}^n (\xi_i-\mu_i) u_i \Big) \cdot h \Big( x e - \sum_{i=1}^n (\xi_i-\mu_i) u_i \Big) \right].
\end{align*}
where $p_{\emptyset}$ is the average of the polynomials in the interlacing family $\mathcal{P}$ in Definition \ref{def:interlacing-family-kls20}.

Finally, by Corollary~\ref{cor:kls-interlacing}, we conclude that there exists $s_1,\dots,s_n\in \supp(\xi_1)\times \cdots \times \supp(\xi_n)$ such that
\begin{align*}
    \left\|\sum_{i=1}^n (s_i-\mu_i) u_i\right\|_h\leq 4.
\end{align*}
Hence we have
\begin{align*}
    \left\|\sum_{i=1}^n (s_i-\mu_i) v_i\right\|_h\leq 4\sigma.
\end{align*}
\end{proof}
\section{Hyperbolic Extension of Kadision-Singer for Strongly Rayleigh }\label{sec:ag}
In this section, we prove Theorem \ref{thm:ag14_ours}. We restate the theorem as follows:

\begin{theorem}[Formal statement of Theorem \ref{thm:ag14_ours}]\label{thm:ag14_ours_formal}
Let $h\in \R[x_1,\dots,x_m]$ denote a hyperbolic polynomial with respect to hyperbolic direction $e\in \Gamma_{++}^h$. 
Let $\mu$ be a homogeneous strongly Rayleigh probability distribution on $[n]$ such that the marginal probability of each element is at most $\epsilon_1$, and let $v_1, \cdots, v_n \in \Gamma_{+}^h$ be $n$ vectors in isotropic positions,
\begin{align*}
    \sum_{i=1}^n v_i  = e,
\end{align*}
such that for all $i \in [n]$, 
\begin{align*}
    \rank_h(v_i) \leq 1 \text{, and } \| v_i \|_h \leq \epsilon_2.  
\end{align*}
Then
\begin{align*}
    \Pr_{S \sim \mu} \left[ \Big\| \sum_{i \in S} v_i \Big\|_h \leq4 (\epsilon_1+ \epsilon_2) + 2 (\epsilon_1+\epsilon_2)^2 \right] >0.
\end{align*}
\end{theorem}

We will introduce the preliminary facts in Section \ref{sec:ag_prelim}. 
In Section \ref{sec:ag_interlacing}, we define the family of hyperbolic characteristic polynomials and show that it forms an interlacing family. Therefore, it remains to upper-bound the largest root of the average of the interlacing polynomial family, i.e. the mixed hyperbolic characteristic polynomial.  We remark that we require a lemma that will be proved later in Section \ref{sec:ag_expectation}.

In Section \ref{sec:ag_expectation}, we reduce the problem of upper-bounding the largest root of the mixed hyperbolic characteristic polynomial to an easier task of upper-bounding the largest root of a multivariate polynomial.  
In Section \ref{sec:ag_bound_root}, we upper bound the largest root of the multivariate polynomial using the multivariate barrier method. 
Finally, in Section \ref{sec:ag-proof-main}, we prove Theorem \ref{thm:ag14_ours_formal}.

\subsection{Preliminaries}\label{sec:ag_prelim}
In this section, we present some preliminary results on strongly Rayleigh distributions.

\begin{definition}[Generating polynomial of probability distribution]
Let $\mu: 2^{[n]} \rightarrow \R_{\geq 0}$ be a probability distribution. For a random variable $X \sim \mu$, the generating polynomial of $\mu$ is defined as
follows:
\begin{align*}
    g_{\mu}(z_1,\dots,z_n)=\E\left[z^X\right]=\sum_{S\subseteq [n]} \Pr[X=S]z^S.
\end{align*}
\end{definition}

\begin{definition}[Strongly Rayleigh distribution]
Let $\mu: 2^{[n]} \rightarrow \R_{\geq 0}$ be a probability distribution and $g_\mu$ be its generating polynomial. We say $\mu$ is strongly Rayleigh (SR) if $g_\mu$ is a real stable polynomial.

Moreover, we say $\mu$ is $d_\mu$-homogeneous strongly Rayleigh if $g_{\mu}$ is $d_{\mu}$-homogeneous real stable.
\end{definition}

We provide two facts about the generating polynomials of Strongly Rayleigh distributions.

\begin{fact}[Marginals of homogeneous SR distributions]\label{fac:marginal_polynomial_SR}
Let $\mu:2^{[n]} \rightarrow \R_{\geq 0}$ be a $d_\mu$-homogeneous SR distribution with generating polynomial $g_\mu$. For $1\leq k\leq n$, consider the marginal distribution on the first $k$ elements $\mu_k: 2^{[k]} \rightarrow \R_{\geq 0}$ such that
\begin{align*} 
    \mu_k(S)=\Pr_{T\sim \mu}[T\cap [k]=S]\quad \forall S\subseteq [k]. 
\end{align*}
Then, for all $S\subseteq [k]$,
\begin{align}\label{eq:marginal_sr}
    \mu_k(S) = x^{|S|-d_\mu}\cdot \prod_{i\in S} \partial_{z_i} \prod_{i\in [k]\backslash S}(1-x\partial_{z_i})g_{\mu}(x\mathbf{1} + z)\Bigg|_{z=0}.
\end{align}
In particular,
\begin{align*}
    \mu(S) = x^{|S|-d_\mu}\cdot \prod_{i\in S} \partial_{z_i} \prod_{i\in [n]\backslash S}(1-x\partial_{z_i})g_{\mu}(x\mathbf{1} + z)\Bigg|_{z=0}.
\end{align*}
\end{fact}

\begin{remark}
We note that the dummy variable $x$ will be cancelled in the RHS of Eqn.~\eqref{eq:marginal_sr}, and hence both sides are numbers. 
\end{remark}
\begin{proof}
Note that $g_\mu$ can be written as
\begin{align*}
    g_{\mu}(x{\bf 1}+z) = f(z_2,\dots,z_n)(x+z_1) + g(z_2,\dots,z_n),
\end{align*}
where 
\begin{align*}
    f(z_2,\dots,z_n):=&~ \sum_{\substack{S\subseteq [n]\backslash\{1\}, \\ |S| = d_\mu - 1}}\mu(S\cup \{1\})\prod_{i\in S}(x+z_i),\quad \text{and}\\
    g(z_2,\dots,z_n):=&~ \sum_{\substack{S\subseteq [n]\backslash\{1\}, \\ |S| = d_\mu}}\mu(S)\prod_{i\in S}(x+z_i).
\end{align*}
First, if $k=1$. We have
\begin{align*}
    \mu_1(\{1\})=&~ \Pr_{T\sim \mu}[1\in T]
    = x^{-d_\mu+1}\cdot f(z_2,\dots,z_n)\Big|_{z=0}
    = x^{-d_\mu+1}\cdot \partial_{z_1} g_{\mu}(x{\bf 1}+z)\Big|_{z=0}.
\end{align*}
Also,
\begin{align*}
    \mu_1(\emptyset)=&~ \Pr_{T\sim \mu}[1\not\in T]
    = x^{-d_\mu}\cdot g(z_2,\dots,z_n)\Big|_{z=0}
    = x^{-d_\mu}\cdot (1-x\partial_{z_1}) g_{\mu}(x{\bf 1}+z)\Big|_{z=0}.
\end{align*}
So, Eqn.~\eqref{eq:marginal_sr} holds for $k=1$.

For the cases where $k>1$s, we can prove by induction on $k$.
Suppose \eqref{eq:marginal_sr} holds for $1, \cdots, k-1$. Let $S\subseteq [k]$ be any subset of $[k]$.
If $k\in S$, then consider $\partial_{z_k}g_{\mu}(x{\bf 1} + z)$, which is the generating polynomial of $\mu'$ that restricts $\mu$ to the sets $T\subseteq[n]\backslash \{k\}$ with $\mu(T\cup \{k\})>0$. Let $S':=S\backslash \{k\}$.  
By the induction hypothesis, we have
\begin{align*}
    \mu(S) = \mu'(S') = &~ x^{|S'|- d_{\mu'}}\cdot  \prod_{i\in S'}\partial_{z_i} \prod_{i\in [k-1]\backslash S'}(1-x\partial_{z_i})\partial_{z_k} g_{\mu}(x{\bf 1}+z)\Bigg|_{z=0} \\
    = &~ x^{|S|-d_\mu}\cdot \prod_{i\in S}\partial_{z_i} \prod_{i\in [k]\backslash S}(1-x\partial_{z_i})g_{\mu}(x{\bf 1}+z)\Bigg|_{z=0},
\end{align*}
where the second line follows from $|S'|=|S|-1$ and $d_{\mu'}=d_{\mu}-1$.

If $k\notin S$, then consider $(1-x\partial_{z_k})g_{\mu}(x{\bf 1} + z)\Big|_{z_k=0}$, which is the generating polynomial of $\mu''$ that restricts $\mu$ to the sets $T\subseteq [n]\backslash\{k\}$ with $\mu(T)>0$.
Then by the induction hypothesis, we have
\begin{align*}
    \mu(S)=\mu''(S)= &~ x^{|S|- d_{\mu''}}\cdot  \prod_{i\in S}\partial_{z_i} \prod_{i\in [k-1]\backslash S}(1-x\partial_{z_i})(1-x\partial_{z_k})g_{\mu}(x{\bf 1} + z)\Bigg|_{z=0} \\
    = &~ x^{|S|-d_{\mu}}\cdot \prod_{i\in S}\partial_{z_i} \prod_{i\in [k]\backslash S}(1-x\partial_{z_i})g_{\mu}(x{\bf 1} + z)\Bigg|_{z=0}.
\end{align*}

Hence, Eqn.~\eqref{eq:marginal_sr} holds for all $k\in [n]$, which completes the proof of the fact.
\end{proof}

\subsection{Defining interlacing family of characteristic polynomials}\label{sec:ag_interlacing}

In this section, we consider the following family of polynomials:

\begin{definition}[Interlacing Family of Theorem \ref{thm:ag14_ours_formal}]\label{def:interlacing-family-ag14}
    Let $h\in \R[x_1,\dots,x_m]$ denote a degree-$d$ hyperbolic polynomial with respect to hyperbolic direction $e\in \Gamma_{++}^h$.  Let $\mu : 2^{[n]} \to \R$ be a homogeneous strongly Rayleigh probability distribution. Let $v_1,\dots, v_n\in \Gamma_+^h$ be $n$ vectors such that $\rank_h(v_i) \leq 1$ for all $i\in [n]$. Let $\mathcal{F} = \{S \subseteq [n]: \mu(S) > 0\}$ be the support of $\mu$.
    For any $S\in \mathcal{F}$, let 
    \begin{align}\label{eq:ag-interlacing}
        q_S(x) = \mu(S) \cdot h\left(xe - \sum_{i \in S} v_i\right).
    \end{align}
    Let $\mathcal{Q }$ denote the following family of polynomials:
    \begin{align*}
        \mathcal{Q} := \Big\{ q_{s_1\cdots s_\ell}(s) = \sum_{\substack{t_{\ell+1} ,\cdots,t_{n} \\ (s_1\cdots s_\ell,t_{\ell+1} ,\cdots,t_{n}) \in \mathcal{F}}} q_{s_1,\cdots,s_\ell, t_{\ell+1},\cdots,t_n} : \forall \ell \in [n], (s_1,\cdots,s_\ell) \in \mathcal{F}|_{[\ell]} \Big\}.
    \end{align*}
    where $\mathcal{F}|_{[\ell]}$ is $\mathcal{F}$ restricted to $[\ell]$, and a subset is represented by a binary indicator vector. %
\end{definition}

We show that the family defined above is an interlacing family. We will crucially use this fact to show Theorem \ref{thm:ag14_ours_formal}.

\begin{lemma}[Hyperbolic version of Theorem 3.3 in \cite{ag14}]\label{lem:ag_define_interlacing_family}
    Let $\mathcal{Q}$ denote the family of polynomials as in Definition \ref{def:interlacing-family-ag14}.
    Then, if the polynomial
    \begin{align*}
        \E_{\xi\sim \mu}\left[h(xe - (\sum_{i=1}^n \xi_i v_i))\right]
    \end{align*}
    is real-rooted in $x$ for any strongly Rayleigh distribution $\mu$, then the polynomial family $\mathcal{Q}$ forms an interlacing family.
\end{lemma}
The proof of this lemma is the same as that of Theorem 3.3 in \cite{ag14}. 

From Lemma \ref{lem:rootbound} and Fact \ref{fact:hyperbolic-norm-max-root}, we obtain the following corollary:
\begin{corollary}\label{cor:ag_largest_root}
    Let $h\in \R[x_1,\dots,x_m]$ denote a degree-$d$ hyperbolic polynomial with respect to hyperbolic direction $e\in \Gamma_{++}^h$.  Let $\mu : 2^{[n]} \to \R$ be a homogeneous strongly Rayleigh probability distribution. Let $v_1,\dots, v_n\in \Gamma_+^h$ be $n$ vectors such that $\rank_h(v_i) \leq 1$ for all $i\in [n]$. Let $\mathcal{F} = \{S \subseteq [n]: \mu(S) > 0\}$ be the support of $\mu$.
    Then there exists $S\in \mathcal{F}$, such that the hyperbolic norm $\|\sum_{i \in S} v_i\|_h$ equals to the largest root of  $q_S$, and is upper bounded by the largest root of 
    \begin{align*}
        q_{\emptyset} = \E_{\xi\sim \mu}\left[h(xe - (\sum_{i=1}^n \xi_i v_i))\right].
    \end{align*}
\end{corollary}
\begin{proof}
    By Lemma \ref{lem:ag14_theorem_3.1} (we remark that this lemma does not depend on the results in this section. We will prove Lemma \ref{lem:ag14_theorem_3.1} in Section \ref{sec:ag_expectation}), the mixed characteristic polynomial 
    \begin{align*}
    \E_{\xi\sim \mu}\left[h(xe - (\sum_{i=1}^n \xi_i v_i))\right]
    \end{align*}
    is real-rooted. Then by Lemma \ref{lem:ag_define_interlacing_family}, the polynomial family $\mathcal{Q}$ (see Definition \ref{def:interlacing-family-ag14}) is an interlacing family. Therefore, by Lemma \ref{lem:interlacing_ub_root}, there exists a subset $S\in \mathcal{F}$, such that the largest root of $q_S$, is upper bounded by the largest root of 
    \begin{align*}
        q_{\emptyset} = \E_{\xi\sim \mu}\left[h(xe - (\sum_{i=1}^n \xi_i v_i))\right].
    \end{align*}
    The corollary then follows from Fact \ref{fact:hyperbolic-norm-max-root}.
\end{proof}

\subsection{From mixed characteristic polynomial to multivariate polynomial}\label{sec:ag_expectation}

In this section, we want to show that the mixed characteristic polynomial
\begin{align*}
    \E_{\xi\sim \mu}\left[h(x^2e - (\sum_{i=1}^n \xi_i v_i))\right]
\end{align*}
has roots equals to the roots of the following multivariate polynomial after taking $z_1=\cdots = z_n = 0$. 
\begin{align*}
    \prod_{i=1}^n (1-\frac{1}{2}\partial^2_{z_i})\Big(h(xe+\sum_{i=1}^nz_iv_i)g_\mu(x{\bf 1}+z)\Big)
\end{align*}
The largest root of this multivariate polynomial is relatively easy to upper-bound using barrier argument. We will describe the details in Section \ref{sec:ag_bound_root}.

The main lemma of this section is as follows:

\begin{lemma}[Hyperbolic version of Theorem 3.1 in \cite{ag14}]\label{lem:ag14_theorem_3.1}
Let $h\in \R[x_1,\cdots,x_m]$ be a degree-$d$ hyperbolic polynomial with hyperbolic direction $e\in \Gamma_{++}^h$. Let $\mu : 2^{[n]} \to \R$ be a $d_\mu$-homogeneous strongly Rayleigh probability distribution with generating polynomial $g_{\mu} \in \R[z_1,\cdots,z_n]$. Let $v_1,\dots, v_n\in \Gamma_+^h$ be $n$ vectors such that $\rank_h(v_i) \leq 1$ for all $i\in [n]$. Then, we have
\begin{align}
x^{d_\mu}\cdot \E_{\xi\sim \mu}\left[h(x^2e - (\sum_{i=1}^n \xi_i v_i))\right]=  x^d \cdot \prod_{i=1}^n (1-\frac{1}{2}\partial^2_{z_i})\Big(h(xe+\sum_{i=1}^nz_iv_i)g_\mu(x{\bf 1}+z)\Big)\Bigg|_{z=0}. \label{eqn:ag_expected_hyp}
\end{align}
Moreover, 
\begin{align*}
    \E_{\xi\sim \mu}\left[h(xe - (\sum_{i=1}^n \xi_i v_i))\right]
\end{align*}
is real-rooted in $x$. %
\end{lemma}

\begin{proof}

First, we rewrite the left-hand-side of \eqref{eqn:ag_expected_hyp} as the expectation over $T\sim \mu$ as an expectation over all indicator $\xi$ of the subsets of $[n]$:
\begin{align}
    &~ x^{d_\mu}\cdot 2^{-n}  \cdot \E_{T\sim \mu}\left[h(x^2e - (\sum_{i \in T} v_i))\right]\notag\\
    = &~ x^{d_\mu}\cdot 2^{-n}  \cdot \sum_{\xi \in \{0,1\}^n} \left( h(x^2e - (\sum_{i=1}^n \xi_i v_i)) \cdot \mu(\xi) \right) \notag\\
    = &~ x^{d_\mu} \cdot \E_{\xi\sim \{0,1\}^n}\left[ h(x^2e-\sum_{i=1}^n \xi_iv_i) \cdot \mu(\xi)\right] \notag\\
     = &~ x^{d_\mu}\cdot \E_{\xi\sim \{0,1\}^n}\left[ h(x^2e -\sum_{i=1}^n \xi_iv_i)\cdot x^{\sum_{i=1}^n \xi_i - d_\mu}\cdot 
    \prod_{i=1}^n \big(\xi_i \partial_{z_i} + (1-\xi_i)(1-z_i\partial_{z_i}) \big)g(x{\bf 1}+z)\Bigg|_{z=0}\right] 
\end{align}
where in the third step, we let $\xi \in \{0, 1\}^n$ be a random bit string uniformly sampled from $\{0, 1\}^n$.
In the last step we use Fact~\ref{fac:marginal_polynomial_SR} and the fact that $\prod_{i=1}^n \big( \xi_i \partial_{z_i} + (1 - \xi_i)(1 - z_i \partial_{z_i}) \big) = \prod_{i \in [n]: \xi_i = 1} \partial_{z_i} \prod_{i \in [n]: \xi_i = 0} (1 - z_i \partial_{z_i})$.
Setting $g_2\in \R[t]$ as
\begin{align*}
    g_2(t) := x^{\sum_{i=2}^n \xi_i}\cdot \prod_{i=2}^n \big(\xi_i \partial_{z_i} + (1-\xi_i)(1-x\partial_{z_i})\big) g(t,x+z_2,x+z_3,\cdots,x+z_n)\Big|_{z_2,\dots,z_n=0}
\end{align*}
and $x_2=x^2e-\sum_{i=2}^n \xi_i v_i$, we can simplify the above equation as
\begin{align}
    &~ x^{d_\mu}\cdot 2^{-n}  \cdot \E_{T\sim \mu}\left[h(x^2e - (\sum_{i \in T} v_i))\right]\notag\\
     = &~ \frac{1}{2}\E_{\xi_2, \cdots, \xi_n \sim \{0,1\}^{n-1}}\left[ h(x_2-v_1)x\partial_{ z_1}g_2(x+z_1)+h(x_2)(1-x\partial{z_1})g_2(x+z_1)\Bigg|_{z_1=0} \right]\notag \\
     = &~ \frac{1}{2}\E_{\xi_2, \cdots, \xi_n \sim \{0,1\}^{n-1}}\left[ (1-\partial_{z_1})h(x_2+z_1v_1)x\partial_{ z_1}g_2(x+z_1)+h(x_2)(1-x\partial_{z_1})g_2(x+z_1)\Bigg|_{z_1=0} \right]. \label{eq:induction_exp_ag14}
\end{align}

Now, we can expand the term inside the expectation of Eqn.~\eqref{eq:induction_exp_ag14}, and get that
\begin{align*}
    &~ (1-\partial_{z_1})h(x_2+z_1v_1)x\partial_{ z_1}g_2(x+z_1)+h(x_2)(1-x\partial_{z_1})g_2(x+z_1)\Bigg|_{z_1=0} \\
    = &~ xh(x_2+z_1v_1)(\partial_{z_1}g_2(x+z_1)) - x(\partial_{z_1} h(x_2+z_1v_1))(\partial_{z_1}g_2(x +z_1))\\
    &~ +h(x_2)g_2(x+z_1)-h(x_2)(x\partial_{z_1}g_2(x+z_1))\Bigg|_{z_1=0}\\
    = &~ xh(x_2)(Dg_2)(x)-x(D_{v_1}h)(x_2)(Dg_2)(x)+h(x_2)g_2(x)-xh(x_2)(Dg_2)(x)\\
    = &~ h(x_2)g_2(x)-xD_{v_1}h(x_2)(Dg_2)(x) \\
    = &~ (1-\frac{x}{2}\partial^2_{z_1})\Big( h(x_2+z_1v_1)g_2(x+z_1) \Big)\Bigg|_{z_1=0},
\end{align*}
where the last step follows from $\rank_h(v_1)\leq 1$ and $\deg(g_2)\leq 1$.

Therefore, the left-hand-side of \eqref{eqn:ag_expected_hyp} equals to
\begin{align}
    &~ x^{d_\mu}\cdot 2^{-n} \cdot \E_{\xi\sim \mu}\left[h(x^2 e - (\sum_{i=1}^n \xi_i v_i))\right] \nonumber \\
    = &~ \frac{1}{2}\E_{\xi_2,\dots,\xi_n}\left[ (1-\frac{x}{2}\partial^2_{z_1})\Big( h(x_2+z_1v_1)g_2(x + z_1) \Big)\Bigg|_{z_1=0} \right] \nonumber\\
    = &~ \frac{1}{2} (1-\frac{x}{2}\partial^2_{z_1})\left( \E_{\xi_2,\dots,\xi_n}\left[h(xe-\sum_{i=2}^n \xi_i v_i+z_1v_1)g_2(x+z_1)\right] \right)\Bigg|_{z_1=0}.  \label{eq:ag_expected_hyp_recursive}
\end{align}

If we repeat this process for $n$ times, we will finally get
\begin{align}
x^{d_\mu} \cdot \E_{\xi\sim \mu}\left[h(x^2e - (\sum_{i=1}^n \xi_i v_i))\right]= &~ \prod_{i=1}^n (1-\frac{x}{2}\partial^2_{z_i})\Big(h(x^2e+\sum_{i=1}^nz_iv_i)g_\mu(x{\bf 1}+z)\Big)\Bigg|_{z=0}. \label{eq:ag_expected_hyp_1}
\end{align}

Now we show that the right-hand-side of \eqref{eq:ag_expected_hyp_1} equals to the right-hand-side of \eqref{eqn:ag_expected_hyp}.
First, we expand the product of partial operator $\prod_{i=1}^n (1 - \frac{x}{2} \partial^2_{z_i})$ and get that
\begin{align*}
    \prod_{i=1}^n (1-\frac{x}{2}\partial^2_{z_i})\Big(h(x^2e+\sum_{i=1}^nz_iv_i)g_\mu(x{\bf 1}+z)\Big)\Bigg|_{z=0}
    = &~ \sum_{T\subseteq [n]}(-\frac{x}{2})^{|T|}\partial_{z^{T}}^2\Big(h(x^2e+\sum_{i=1}^nz_iv_i)g_\mu(x{\bf 1}+z)\Big)\Bigg|_{z=0}.
\end{align*}
For any $T\subseteq [n]$ with $|T|=k$, we have
\begin{align*}
    (-\frac{x}{2})^{k}\partial_{z^{T}}^2\Big(h(x^2e+\sum_{i=1}^nz_iv_i)g_\mu(x{\bf 1}+z)\Big)\Bigg|_{z=0} = &~ (-\frac{x}{2})^{k} \cdot 2^{k}\cdot \left(\prod_{i\in T}D_{v_i} \right)h(x^2e)\cdot g_\mu^{(T)}(x{\bf 1}),
\end{align*}
where $g_\mu^{(T)}(x{\bf 1})=\prod_{i\in T}\partial_{z_i} g_\mu(x{\bf 1}+z)\Big|_{z=0}$.

Since $h$ is $d$-homogeneous, we know that $(\prod_{i\in T}D_{v_i})h$ is $(d-k)$-homogeneous. Hence, we get that
\begin{align*}
    (-\frac{x}{2})^{k}\partial_{z^{T}}^2\Big(h(x^2e+\sum_{i=1}^nz_iv_i)g_\mu(x{\bf 1}+z)\Big)\Bigg|_{z=0} = &~(-\frac{x}{2})^{k} \cdot 2^{k} \cdot x^{d-k}\cdot \left(\prod_{i\in T}D_{v_i} \right)h(xe)\cdot g_\mu^{(T)}(x{\bf 1})\\
    = &~ x^d\cdot (-1)^k\cdot \left(\prod_{i\in T}D_{v_i} \right)h(xe)\cdot g_\mu^{(T)}(x{\bf 1})\\
    = &~ x^d \cdot (-\frac{1}{2})^{k}\partial_{z^T}^2\left(h(xe+\sum_{i=1}^n z_i v_i)g_\mu(x{\bf 1}+z)\right)\Bigg|_{z=0}.
\end{align*}
Therefore,
\begin{align*}
    x^{d_\mu}\cdot \E_{\xi\sim \mu}\left[h(x^2e - (\sum_{i=1}^n \xi_i v_i))\right]= &~ \sum_{T\subseteq [n]} x^d \cdot (-\frac{1}{2})^{k}\partial_{z^T}^2\left(h(xe+\sum_{i=1}^n z_i v_i)g_\mu(x{\bf 1}+z)\right)\Bigg|_{z=0}\\
    = &~ x^d \cdot \prod_{i=1}^n (1-\frac{1}{2}\partial^2_{z_i})\Big(h(xe+\sum_{i=1}^nz_iv_i)g_\mu(x{\bf 1}+z)\Big)\Bigg|_{z=0},
\end{align*}
which completes the proof of Eqn. (\ref{eqn:ag_expected_hyp}).

Finally, we show that  Eqn. (\ref{eqn:ag_expected_hyp}) is real-rooted in $x$. Since $e\in \Gamma_{++}^h$ and $v_1,\dots,v_n\in \Gamma^h_+$, by Claim~\ref{clm:hyper_linear_restriction}, we get that $h(xe+\sum_{i=1}^nz_iv_i)$ is real-stable. Furthermore, since $g_\mu(x{\bf 1}+z)$ is real-stable, by Fact \ref{fac:closure_real_stable}, 
\begin{align*}
    h(xe+\sum_{i=1}^nz_iv_i)g_\mu(x{\bf 1}+z)
\end{align*}
is also real-stable.
Then by Fact \ref{fac:closure_real_stable} again, we get that
\begin{align*}
    \prod_{i=1}^n (1-\frac{1}{2}\partial^2_{z_i})\Big(h(xe+\sum_{i=1}^nz_iv_i)g_\mu(x{\bf 1}+z)\Big)\Bigg|_{z=0}
\end{align*}
is real-stable. 
By Fact \ref{fac:uni_real_stable}, it implies that it is real-rooted in $x$. Equivalently, 
\begin{align*}
    q_\emptyset(x^2)=\E_{\xi\sim \mu}\left[h(x^2e - (\sum_{i=1}^n \xi_i v_i))\right]
\end{align*}
is real-rooted in $x$. Then, it is easy to see that 
\begin{align*}
    q_\emptyset(x)=\E_{\xi\sim \mu}\left[h(xe - (\sum_{i=1}^n \xi_i v_i))\right]
\end{align*}
is also real-rooted in $x$, since a complex root of $q_\emptyset(x)$ implies a complex root of $q_\emptyset(x^2)$.

The lemma is then proved.
\end{proof}

\subsection{Applying barrier argument to bound the largest root of multivariate polynomial}\label{sec:ag_bound_root}

In this section, we upper bound the largest root of the following multivariate polynomial:
\begin{align*}
    \prod_{i=1}^n (1-\frac{1}{2}\partial^2_{z_i})\Big(h(xe+\sum_{i=1}^nz_iv_i)g_\mu(x{\bf 1}+z)\Big)
\end{align*}
using the real-stable version of the barrier method in \cite{ag14}.

\begin{lemma}[Hyperbolic version of Theorem 4.1 in \cite{ag14}]\label{lem:ag_bound_root}
Let $h\in \R[x_1,\cdots,x_m]$ denote a degree-$d$ hyperbolic polynomial with hyperbolic direction $e\in\Gamma_{++}^h$. Let $v_1, \cdots, v_n \in \Gamma^h_+$ be $n$ vectors such that $\sum_{i=1}^n v_i=e$ and $\tr_h[ v_i ]\leq \epsilon_2$ for all $i\in [n]$. Let  $\mu : 2^{[n]} \rightarrow \R_{\geq 0}$ be a $d_{\mu}$-homogeneous strongly Rayleigh probability distribution such that the marginal probability of each element $i \in [n]$ is at most $\epsilon_1$. Then, all the roots of 
\begin{align*}
    \prod_{i=1}^n (1-\frac{1}{2}\partial^2_{z_i})\Big(h(xe+\sum_{i=1}^nz_iv_i)g_\mu(x{\bf 1}+z)\Big) \in \R[x,z_1,\cdots,z_n]
\end{align*}
lie below $(\sqrt{4\epsilon +2\epsilon^2},0,\cdots,0) \in \R^{n+1}$, where $\epsilon = \epsilon_1 + \epsilon_2$.
\end{lemma}

\begin{proof}
Let $Q\in \R[x,z_1,\dots,z_n]$ be an $(n+1)$-variate polynomial:
\begin{align*}
    Q(x,z):=h(xe+\sum_{i=1}^nz_iv_i)g_\mu(x{\bf 1}+z).
\end{align*}
We have already proved in Lemma~\ref{lem:ag14_theorem_3.1} that $Q(x,z)$ is real-stable.

For $0<t<\alpha$ where $\alpha=\alpha(t)$ is a parameter to be chosen later, we have
\begin{align*}
    Q(\alpha, -t{\bf 1}) = &~ h(\alpha e -t\sum_{i=1}^n v_i)g_{\mu}((\alpha-t){\bf 1})\\
    = &~ h((\alpha-t) e)g_{\mu}((\alpha-t){\bf 1})\\
    = &~ (\alpha - t)^{d+d_\mu}h(e)g_\mu({\bf 1})\\
    > &~ 0,
\end{align*}
where the second step follows from $\sum_{i=1}^n v_i = e$ and the last step follows from $\alpha > t$, $h(e)>0$ and $g_\mu({\bf 1})=1$.
This implies that $(\alpha, -t{\bf 1}) \in \Ab_Q$.

We can upper bound $\Phi_Q^i(\alpha, -t{\bf 1})$ as follows.
\begin{align*}
    \Phi_Q^i(\alpha, -t{\bf 1}) = &~ \frac{\partial_{z_i} Q}{Q}\Bigg|_{x=\alpha, z=-t{\bf 1}}\\
    = &~ \frac{(\partial_{z_i} h(xe+\sum_{i=1}^n z_i v_i))g_\mu(x{\bf 1}+z)+h(xe+\sum_{i=1}^n z_i v_i)(\partial_{z_i}g_\mu({\bf 1}+z))}{h(xe+\sum_{i=1}^nz_iv_i)g_\mu(x{\bf 1}+z)}\Bigg|_{x=\alpha, z=-t{\bf 1}}\\
    = &~ \frac{\partial_{z_i} h(xe+\sum_{i=1}^n z_i v_i)}{h(xe+\sum_{i=1}^n z_i v_i)}+\frac{\partial_{z_i}g_\mu(x{\bf 1}+z)}{g_\mu(x{\bf 1}+z)}\Bigg|_{x=\alpha, z=-t{\bf 1}}\\
    = &~ \frac{D_{v_i} h(xe+\sum_{i=1}^n z_i v_i)}{h(xe+\sum_{i=1}^n z_i v_i)}+\frac{\partial_{z_i}g_\mu(x{\bf 1}+z)}{g_\mu(x{\bf 1}+z)}\Bigg|_{x=\alpha, z=-t{\bf 1}}\\
    = &~ \frac{D_{v_i}h(\alpha e - te)}{h(\alpha e - te)}+\frac{(\alpha-t)^{d_\mu-1}\cdot \Pr_{S\sim \mu}[i\in S]}{(\alpha-t)^{d_\mu}}\\
    \leq  &~ \frac{\tr[v_i]}{\alpha - t}+\frac{\epsilon_1}{\alpha-t}\\
    \leq &~ \frac{\epsilon_1 + \epsilon_2}{\alpha - t},
\end{align*}
where the first inequality follows from Fact~\ref{fac:related_to_D_v_h}.

Let $\epsilon:=\epsilon_1 + \epsilon_2$. By choosing $\alpha = 2t = \sqrt{4\epsilon+2\epsilon^2}$, we get that 
\begin{align*}
    \Phi_Q^i(\alpha, -t{\bf 1})\leq \frac{\epsilon}{\sqrt{\epsilon+\epsilon^2/2}}= \frac{\sqrt{2}}{\sqrt{1+2/\epsilon}}<\sqrt{2}.
\end{align*}

Then, by Lemma~\ref{lem:real_stable_cone}, we know that $(\alpha, -t{\bf 1})\in \Ab_{(1-\frac{1}{2}\partial_{z_i}^2)Q}$ for any $i\in [m]$.

Furthermore, 
\begin{align*}
    \frac{1}{t}\Phi_Q^i(\alpha, -t{\bf 1}) + \frac{1}{2}\Phi_Q^i(\alpha, -t{\bf 1})^2 \leq &~ \frac{\epsilon}{t^2} +\frac{1}{2}\frac{\epsilon^2}{t^2}=1.
\end{align*}

By the second part of Lemma~\ref{lem:real_stable_cone}, we have for all $j\in [m]$,
\begin{align*}
    \Phi^j_{(1-\frac{1}{2}\partial_{z_i}^2)Q} (\alpha, -t{\bf 1} + t{\bf 1}_i)\leq \Phi_Q^j(\alpha, -t{\bf 1}).
\end{align*}

By a similar induction process like in the proof of Lemma~\ref{lem:kls_bound_root}, we have
\begin{align*}
    (\alpha, -t{\bf 1} + \sum_{i=1}^n t {\bf 1}_i) &= (\sqrt{4\epsilon + 2\epsilon^2}, 0,0,\dots,0) \\
    &\in \Ab_{\prod_{i=1}^n (1 - \partial_{z_i}^2 / 2) Q}
\end{align*}
i.e. $(\sqrt{4\epsilon + 2\epsilon^2}, 0,0,\dots,0)$ lies above the roots of 
\begin{align*}
    \prod_{i=1}^n \Big( 1 - \frac{1}{2} \frac{ \partial^2 }{ \partial z_i^2 } \Big)  \left[ h \Big( x e + \sum_{i=1}^n \xi_i v_i \Big) \cdot g_\mu(x{\bf 1}+z) \right]
\end{align*}
as desired.
\end{proof}

\subsection{Combining together: proof of Theorem \ref{thm:ag14_ours_formal}}\label{sec:ag-proof-main}

Now we can combine the results from the previous section and prove Theorem \ref{thm:ag14_ours_formal}:

\begin{proof}[Proof of Theorem \ref{thm:ag14_ours_formal}]

Let $\mathcal{F}$ be the support of $\mu$.

By Lemma~\ref{lem:ag_bound_root} and restricting to $z_i = 0$ for all $i\in [n]$, we have that $\sqrt{4(\epsilon_1 + \epsilon_2) + 2(\epsilon_1 + \epsilon_2)^2}$
lies above the largest root of the univariate polynomial
\begin{align*}
    \prod_{i=1}^n (1-\frac{1}{2}\partial^2_{z_i})\Big(h(xe+\sum_{i=1}^nz_iv_i)g_\mu(x{\bf 1}+z)\Big)\Bigg|_{z=0}.
\end{align*}

We then conclude by Lemma \ref{lem:ag14_theorem_3.1} that $\sqrt{4(\epsilon_1 + \epsilon_2) + 2(\epsilon_1 + \epsilon_2)^2}$
upper bounds the largest root of 
\begin{align*}
    \E_{\xi\sim \mu}\left[h(x^2e - (\sum_{i=1}^n \xi_i v_i))\right].
\end{align*}
Therefore, the largest root of
\begin{align*}
    q_{\emptyset} = \E_{\xi\sim \mu}\left[h(xe - (\sum_{i=1}^n \xi_i v_i))\right]
\end{align*}
is upper bounded by $4(\epsilon_1 + \epsilon_2) + 2(\epsilon_1 + \epsilon_2)^2$,
where $q_{\emptyset}$ is the average of the polynomials in the interlacing family $\mathcal{Q}$ in Definition \ref{def:interlacing-family-ag14}.

Finally, by Corollary \ref{cor:ag_largest_root}, there exists $S \subseteq[n]$ in the support of $\mu$, such that
\begin{align*}
    \Big\|\sum_{i\in S} v_i\Big\|_h \leq 4(\epsilon_1 + \epsilon_2) + 2(\epsilon_1 + \epsilon_2)^2 .
\end{align*}

\end{proof}

\section{Sub-Exponential Algorithms }\label{sec:sub-exp-alg}

\subsection{Definitions}

\begin{definition}[$k$-th symmetric polynomials]\label{def:k-symmetric-polynomial}
    For any $n \in \Z_+, k \leq n$, let $e_k \in \R[z_1,\cdots,z_n]$ denote the $k$-th elementary symmetric polynomial defined as
    \begin{align*}
        e_k(z_1,\cdots,z_n) = \sum_{T\in \binom{[n]}{k}} \prod_{i\in T} z_i.
    \end{align*}
\end{definition}

\begin{definition}[$k$-th power sum polynomials]\label{def:k-power-sum-polynomial}
    For any $n,k \in \Z_+$, let $p_k \in \R[z_1,\cdots,z_n]$ denote the $k$-th power sum polynomial defined as $p_k(z_1,\cdots,z_n) = \sum_{i=1}^n z_i^k$.
\end{definition}

\begin{fact}[Vieta's formulas]\label{fact:vieta-formula}
    Let $f \in \R[x]$ be any degree $n$ monic variate polynomial defined as $f(x)  = x^n + c_1 x^{n-1} + \cdots + c_n$. Then for any $k\in [n]$,
    \begin{align*}
        c_k = (-1)^k e_k(\lambda_1,\cdots,\lambda_n).
    \end{align*}
    Where $\lambda_1,\cdots,\lambda_n \in \C$ are the roots of $f$, and $e_k$ is the $k$-th elementary symmetric polynomial defined in Definition \ref{def:k-symmetric-polynomial}.
\end{fact}

\begin{fact}[Newton's identities]\label{fact:newton-identity}
    For any $n \in \Z_+, k\in [n]$, let $e_k, p_k \in \R[z_1,\cdots,z_n]$ be defined in Definition \ref{def:k-symmetric-polynomial} and Definition \ref{def:k-power-sum-polynomial} respectively. Then there exists an $O(k^2)$-time algorithm %
    $\textsc{ElemToPower}(k,e_1,\cdots,e_k)$, such that given any $k \in [n]$, and any $e_1 = e_1(x),\cdots,e_k = e_k(x)$ for some fixed $x\in \R^n$, outputs $p_k = p_k(x)$.
\end{fact}

\subsection{Algorithm to approximate the largest root}\label{sec:alg-largest-root}

\begin{algorithm}[ht]\caption{Algorithm approximating the largest root.}\label{alg:alg-max-root}
\begin{algorithmic}[1]
\Procedure{\textsc{MaxRoot}}{$n\in \Z_+,k\in [n], f\in \R[x], c_1,\cdots,c_k \in \R$} 
    \State \textbf{Precondition:} $k\leq n$, $c_1,\cdots,c_k \in \R$ are the top-$k$ coefficients of a degree $n$ real-rooted monic-variate polynomial $f \in \R[x]$.
    \State \textbf{Output:} Return an approximate of the largest root of $f$.
    \For{$i=1$ to $k$}
        \State $e_k \leftarrow c_k \cdot (-1)^k$. 
        \Statex \Comment{$e_k = e_k(\lambda_1,\cdots,\lambda_n)$ is the $k$-th elementary polynomial of the roots of $f$ by Fact \ref{fact:vieta-formula}.}
    \EndFor
    \State $p_k \leftarrow \textsc{ElemToPower}(k,e_1,\cdots,e_k)$ 
    \Statex \Comment{$p_k(\lambda) = \lambda_1^k + \cdots + \lambda_n^k$ is the power sum of the roots of $f$ by Fact \ref{fact:newton-identity}.} 
    \State \textbf{Return} $(p_k)^{1/k}$.
\EndProcedure
\end{algorithmic}
\end{algorithm}

\begin{lemma}\label{lem:alg-max-root-approx}
    Let $f = x^n + c_1 x^{n-1} + \cdots + c_n \in R[x]$ be any degree $n$ real-rooted monic-variate polynomial. 
    Then $\textsc{MaxRoot}(n,k,f,c_1,\cdots,c_k)$ (Algorithm \ref{alg:alg-max-root}) returns an $(n^{1/k})$-approximate of the largest root of $f$ in time $O(k^2 + k)$. 
\end{lemma}
\begin{proof}
    Let $\lambda_1 \geq \lambda_2 \geq \lambda_n \in \R$ denote the roots of $f(x)$. By Fact \ref{fact:vieta-formula} and Fact \ref{fact:newton-identity}, Algorithm \ref{alg:alg-max-root} returns $(p_k)^{1/k}$, where $p_k = \lambda_1^k + \cdots + \lambda_n^k$. Notice that
    \begin{align*}
        \frac{\lambda_1^k + \cdots + \lambda_n^k}{n} \leq \lambda_1^k \leq \lambda_1^k + \cdots + \lambda_n^k
    \end{align*}
    we have
    \begin{align*}
        \lambda_1 \leq (p_k/n)^{1/k} \leq n^{1/k}\lambda_1.
    \end{align*}
\end{proof}

\begin{remark}
    We remark that when $k > \log n$, the approximation factor $n^{1/k}$ is upper bounded by $1 + \frac{\log n}{k}$.
\end{remark}

\subsection{Reducing Kadison-Singer to finding leading coefficients of interlacing polynomial}

We define an oracle that generates the top-$k$ coefficients as follows:
\begin{definition}\label{def:oracle-max-coeff}
Fix a family of degree $n$ monic-variable polynomials $\mathcal{F} = \{f_{s_1,\cdots,s_m}\in \R[x] : s_1,\cdots s_\ell \in S_1 \times \cdots \times S_\ell\}$.
We define oracle $\textsc{MaxCoeff}_{\mathcal{F}}(k, \ell, s_1,\cdots s_\ell )$ as follows: given any $k\in [n], \ell \in [m], s_1,\cdots s_\ell \in S_1 \times \cdots \times S_\ell$ as inputs, and it outputs the top-$k$ coefficients of $f_{s_1 \cdots s_{\ell}}$ %
in $\T_{\mathrm{Coeff}}(\mathcal{F}, k)$ time.
\end{definition}

\begin{lemma}[Theorem 4.4 of \cite{aoss18}]\label{lem:kadison-singer-alg-general}
    Let $S_1,\cdots,S_m$ be finite sets and let $\mathcal{F} = \{f_{s_1 \cdots s_m}(x)\in \R[x] : s_1\in S_1,\cdots,s_m \in S_m\}$ be an interlacing family of degree $n$ real-rooted monic-variate polynomials. Let $\textsc{MaxCoeff}_{\mathcal{F}}(k, \ell, s_1,\cdots s_\ell )$ be the oracle for finding the largest $k$ coefficients defined as Definition \ref{def:oracle-max-coeff}.
    
    Then there is an algorithm $\textsc{KadisonSinger}(\delta, \mathcal{F} = \{f_{s_1 \cdots s_m}(x) : s_1\in S_1,\cdots,s_m \in S_m\})$
    (Algorithm \ref{alg:kadison-singer}) that, given any $\delta > 0$ and $\mathcal{F}$ as input, returns the elements $s_1' \in S_1,\cdots,s_m' \in S_m$,
    such that the maximum root of $f_{s_1' \cdots s_m'}$ is at most $(1+\delta)$ times the maximum root of $f_{\emptyset}$, in time
    \begin{align*}
        \left(\T_{\mathrm{Coeff}}\left(\mathcal{F}, O(\log(n) m \delta^{-1} ) \right) + k^2 \right) \cdot {\cal S}^{\sqrt{m}} \cdot \sqrt{m}, 
    \end{align*}
    where ${\cal S} = \max_{i\in [m]}|S_i|$.
\end{lemma}

\begin{algorithm}[ht]\caption{Algorithm finding approximate solutions of Kadison-Singer problems}\label{alg:kadison-singer}
\begin{algorithmic}[1]
\Procedure{\textsc{KadisonSinger}}{$\delta, \mathcal{F} = \{f_{s_1 \cdots s_m}(x) : s_1\in S_1,\cdots,s_m \in S_m\}$} 
    \State \textbf{Precondition:} $\delta > 0$, $S_1,\cdots,S_m$ are finite sets, $\mathcal{F} = \{f_{s_1 \cdots s_m}(x)\in \R[x] : s_1\in S_1,\cdots,s_m \in S_m\}$ is an interlacing family of degree $n$ real-rooted monic-variate polynomials.
    \State \textbf{Output:} $s_1' \in S_1,\cdots,s_m' \in S_m$, such that the maximum root of $f_{s_1' \cdots s_m'}$ is at most $(1+\delta)$ times the maximum root of $f_{\emptyset}$.
    \State $M \leftarrow \sqrt{m}$
    \State $k \leftarrow \frac{M \log n}{\delta}$
    \State $\lambda_{\max} \leftarrow -\infty$
    \For{$i=0$ to $\frac{m}{M}-1$} \Comment{$\sqrt{m}$ iterations}
        \For{$t_{iM+1} \in S_{iM+1},\cdots, t_{iM+M} \in S_{iM+M}$} \Comment{Brute force search on $\prod_{j=1}^M |S_{iM+j}|$ elements}
            \State $(c_1,\cdots,c_k) \leftarrow \textsc{MaxCoeff}_{\mathcal{F}}(k,M(i+1),s_1,\cdots,s_{iM}, t_{iM+1},\cdots,t_{iM+M} )$
            \State $\lambda \leftarrow \textsc{MaxRoot}(n,k,f_{s_1,\cdots,s_{iM}, t_{iM+1},\cdots,t_{iM+M}}, c_1,\cdots,c_k)$ \label{line:ks-find-root}
            \If{$\lambda \leq \lambda_{\min}$}
                \State $\lambda_{\min} \leftarrow \lambda$
                \State $s_{iM+1}\leftarrow t_{iM+1},\cdots,s_{iM+M} \leftarrow t_{iM+M}$
            \EndIf
        \EndFor
    \EndFor
    \State \textbf{Return} $(s_1,\cdots,s_m)$ 
\EndProcedure
\end{algorithmic}
\end{algorithm}

\begin{proof}
    Notice that Algorithm \ref{alg:kadison-singer} runs in $\frac{m}{M} = \sqrt{M}$ iterations. Inside each iteration, it does a brute force search on at most $\mathcal{Q}^{M} = \mathcal{Q}^{\sqrt{m}}$ elements. For each element $(t_{iM+1},\cdots,t_{iM+M})$, we query the oracles $\textsc{MaxCoeff}_{\mathcal{F}}$ and  $\textsc{MaxRoot}$, which takes $\T_{\mathrm{Coeff}}(\mathcal{F}, k) = \T_{\mathrm{Coeff}}(\mathcal{F}, O(\log(n) m \delta^{-1} )$ and $O(k + k^2)$ time respectively. Therefore, the running time of the algorithm is at most
    \begin{align*}
    \left(\T_{\mathrm{Coeff}}\left(\mathcal{F}, O(\log(n) m \delta^{-1} ) \right) + k^2 \right) \cdot {\cal S}^{\sqrt{m}} \cdot \sqrt{m}.
    \end{align*}
    
    Now we show the correctness of the algorithm by induction on the number of iteration $i$.
    For any $0\leq i\leq \frac{m}{M}-1$,
    suppose we have selected $s_1,\cdots,s_{iM}$ in the previous iterations, and suppose 
    \begin{align*}
        \lambda_{\max}\left( f_{s_1,\cdots,s_{iM}} \right) \leq (1+\frac{\delta}{2M})^{i+1} \cdot \lambda_{\max}(f_{\emptyset})
    \end{align*}
    where $\lambda_{\max}(f)$ denotes the maximum root of the univariate polynomial $f$.
    
    Suppose we fix $s_{iM+1},\cdots,s_{iM+M}$ on the $i$-th iteration. Note that $k = \frac{2M\log n}{\delta} > \log n$, by Lemma \ref{lem:alg-max-root-approx}, Line \ref{line:ks-find-root} of the algorithm returns an $1 + \frac{\log n}{k} = (1 + \frac{\delta}{2M})$-approximation of the largest root of $f_{s_1,\cdots,s_{iM}, t_{iM+1},\cdots,t_{iM+M}}$. Therefore, for any $0\leq i \leq \frac{m}{M}$, we have
    \begin{align*}
        \lambda_{\max}\left( f_{s_1,\cdots,s_{iM+M}} \right) \leq (1+\frac{\delta}{2M}) \cdot \min_{t_{iM+1}\in S_{iM+1},\cdots, t_{iM+M}\in S_{iM+M}} \lambda_{\max}\left( f_{s_1,\cdots,s_Mi,t_{iM+1},\cdots,t_{iM+M}} \right)
    \end{align*}
    Since $\{f_{s_1 \cdots s_m}(x) : s_1\in S_1,\cdots,s_m \in S_m\}$ is an interlacing family, we have
    \begin{align*}
        \min_{t_{iM+1}\in S_{iM+1},\cdots, t_{iM+M}\in S_{iM+M}} \lambda_{\max}\left( f_{s_1,\cdots,s_Mi,t_{iM+1},\cdots,t_{iM+M}} \right) \leq \lambda_{\max}\left( f_{s_1,\cdots,s_{iM}} \right).
    \end{align*}
    This proves the induction hypothesis.
\end{proof}

\begin{remark}
It's important to note that our algorithm runs in $\mathcal{S}^{\tilde{O}(\sqrt{n})}$ time. A paper by \cite{aoss18} provides an algorithm with a faster running time of $\mathcal{S}^{\tilde{O}(\sqrt[3]{n})}$. However, this algorithm is limited to only the $k$-th largest root of \emph{determinantal polynomials}.

Our algorithm faces a challenge in approximating the $k$-th largest root of \emph{hyperbolic polynomials}. In order to achieve an $(1 + \epsilon)$ approximation, our sub-exponential algorithm must run $\tilde{O}(\sqrt{n}\epsilon)$ iterations and search $k = \tilde{O}(\sqrt{n}/\epsilon)$ elements at each time. This ensures that the cumulative error does not exceed $\epsilon$. During each iteration, we have to brute-force over $O(\sqrt{n})$ elements, which results in a $2^{\tilde{O}(n)}$ search time.
\end{remark}

\subsection{Sub-exponential algorithm for Theorem \ref{thm:kls20_ours}}\label{sec:alg-kls20-ours}

In this section, we want to describe the sub-exponential algorithm for constructing Theorem \ref{thm:kls20_ours}.
Let $\mathcal{P}$ denote the interlacing family 
\ifdefined\isarxiv
as defined in Definition \ref{def:interlacing-family-kls20}. 
\else
as follows:
    Let $\xi_1,\dots,\xi_n$ denote $n$ independent random variables with finite supports and $\E[\xi_i]=\mu_i$ for $i\in [n]$. Let $v_1, \dots, v_n \in \Gamma_{+}^h$ be $n$ vectors such that $\mathrm{rank}_h(v_i) \leq 1$ for all $i \in [n]$. 
    For each $s = (s_1,\dots,s_n)$ where $s_i\in \supp(\xi_i)$, let $p_{s} \in \R [x]$ define the following polynomial:
    \begin{align*}
        p_{\bf s}(x) := \left(\prod_{i=1}^n p_{i, s_i}\right)\cdot h\left(x e + \sum_{i=1}^n (s_i-\mu_i)  v_i  \right) \cdot h\left(x e - \sum_{i=1}^n (s_i-\mu_i)  v_i  \right)
    \end{align*}
    where $p_{i, s_i} := \Pr_{\xi_i}[\xi_i = s_i]$.
    Let $\mathcal{P}$ denote the following family of polynomials:
    \begin{align}
        \mathcal{P} := \left\{ p_{(s_1,\cdots, s_\ell)} = \sum_{\substack{t_{\ell+1} ,\cdots,t_{n}:\\t_j\in \supp(\xi_j)~\forall j\in \{\ell+1,\dots,n\}}} p_{(s_1,\cdots,s_\ell, t_{\ell+1},\cdots,t_n)} : ~ \ell \in [n], s_i \in \supp(\xi_i)~\forall i\in [\ell]\right\}.
        \label{def:kls_interlacing}
    \end{align}
\fi

Suppose each $p_\mathbf{s}(x) \in \mathcal{P}$ has degree $d$.
Let $\textsc{MaxCoeff}_\mathcal{P}(k,\ell,s_1,\cdots,s_\ell)$ be the oracle defined in Definition \ref{def:oracle-max-coeff}, i.e. given any $k\in [d]$, $\ell \in [n]$, $(s_1,\cdots,s_\ell) \in \{\pm 1\}^\ell$, outputs the top-$k$ coefficients of $ p_{s_1,\cdots,s_\ell}$ in at most $\T_{\mathrm{Coeff}}(\mathcal{P}, k)$ time. The following lemma states that if $\T_{\mathrm{Coeff}}(\mathcal{P}, k)$ is polynomial in $k$, then we can construct Theorem \ref{thm:kls20_ours} in sub-exponential time:

\begin{corollary}[Sub-exponential algorithm for Theorem \ref{thm:kls20_ours}, formal statement of Proposition \ref{prop:subexpalg-kls20-ours-informal}]\label{cor:subexpalg-kls20-ours}
    Let $h\in \R[x_1,\dots,x_m]$ denote a hyperbolic polynomial with respect to $e\in \Gamma_{++}^h$. Let $u_1, \dots, u_n \in \Gamma_{+}^h$ be $n$ vectors such that
    \begin{align*}
        \sigma = \Big\| \sum_{i=1}^n \tr_h[u_i] u_i \Big\|_h.
    \end{align*}
    Let $\mathcal{P}$ be the interlacing family 
    defined
    \ifdefined\isarxiv
    in Definition \ref{def:interlacing-family-kls20}.
    \else
    in \eqref{eq:kls_interlacing}
    \fi
    Let $\textsc{MaxCoeff}_\mathcal{P}$ be the oracle defined in Definition \ref{def:oracle-max-coeff} with running time $\T_{\mathrm{Coeff}}(\mathcal{P}, k)$.
    Then for any $\delta > 0$, the algorithm $\mathrm{KadisonSinger}(\delta, \mathcal{P})$ (Algorithm \ref{alg:kadison-singer}) returns a sign assignment $(s_1,\cdots,s_n)\in \{\pm 1\}^n$, such that
    \begin{align*}
     \Big\| \sum_{i=1}^n s_i u_i \Big\|_h \leq 4(1+\delta) \sigma
    \end{align*} 
    in time 
    \begin{align*}
         \left(\T_{\mathrm{Coeff}}\left(\mathcal{P}, O(\log(n) m \delta^{-1})\right) + \frac{m \log^2 n}{\delta^2} \right)
         \cdot 2^{\sqrt{m}} \cdot \sqrt{m}.
    \end{align*}
\end{corollary}

\begin{proof}
    For all $i\in [n]$, let $v_i = \frac{u_i}{\sqrt{\sigma}}$. Then we have
    \begin{align*}
        \left\|\sum_{i=1}^n \tr_h[v_i]v_i\right\|_h =  \left\|\sum_{i=1}^n \frac{\tr_h[u_i]u_i}{\sigma}\right\|_h=1.
    \end{align*}
    
    Let $s_1,\cdots,s_m$ denote the output of $\textsc{KadisonSinger}(\sigma, \mathcal{P})$ (Algorithm \ref{alg:kadison-singer}). Then we have
    \begin{align*}
        \|\sum_{i=1}^m s_i v_i\|_h & = \lambda_{\max}(p_{s_1,\cdots,s_m}) \\
        & \leq (1+\delta) \cdot \lambda_{\max}(p_{\emptyset}) & (\text{By Lemma \ref{lem:kadison-singer-alg-general}}) \\
        & \leq 4(1 + \delta)  & 
        (\text{By Lemma \ref{lem:kls_bound_root}})
    \end{align*}
    Therefore we have
    \begin{align*}
        \|\sum_{i=1}^m s_i u_i\|_h \leq 4\sigma(1 + \delta) \leq 8\sigma .
    \end{align*}
\end{proof}

\subsection{Sub-exponential algorithm for Theorem \ref{thm:kls20}}\label{sec:alg-kls20}

In particular, we can explicitly compute the running time of the sub-exponential algorithm for Theorem \ref{thm:kls20}. We first define the interlacing family for Theorem \ref{thm:kls20} with the following statement:

\begin{lemma}[Interlacing family for Theorem \ref{thm:kls20}, Proposition 4.1 and 5.4 of \cite{kls20}]\label{lem:interlacing-family-kls20-original}
    Let $\xi_1,\cdots,\xi_n$ denote $n$ i.i.d. random variables sampled uniformly at random from $\{\pm 1\}$. Let $u_1, \dots, u_n \in \Gamma_{+}^h$ be $n$ vectors such that $\sigma^2 = \|\sum_{i=1}^n (u_i u_i^*)^2\|$
    
    For each $\mathbf{s} \in \{\pm 1\}^n$, let $f_{\mathbf{s}} \in \R[x]$ denote the following polynomial:
    \begin{align*}
        f_{\mathbf{s}}(x) := (\prod_{i=1}^n p_{i,s_i}) \cdot \det \left(x^2 - \left( \sum_{i=1}^n (s_i - \lambda_i)u_i u_i^* \right)^2\right).
    \end{align*}
    where $\forall i\in [n]$, $\lambda_i = \E[\xi_i]$, and $\forall i\in [n], \forall s_i \in \{\pm 1\}$, $p_{i, s_i} = \Pr_{\xi_i}[\xi_i = s_i]$.
    Let $\mathcal{F}$ denote the following family of polynomials:
    \begin{align*}
        \mathcal{F} := \left\{ f_{s_1\cdots s_\ell}(x) = \sum_{t_{\ell+1} \in \{\pm 1\},\cdots,t_{n}\in \{\pm 1\}} f_{s_1,\cdots,s_\ell, t_{\ell+1},\cdots,t_n} : \forall \ell \in [n], s_1,\cdots,s_\ell \in \{\pm 1\}^\ell \right\}.
    \end{align*}
    Then $\mathcal{F}$ is an interlacing family. Moreover, there exists a choice of outcomes $s_1,\cdots,s_n\in \{\pm 1\}^n$, such that
    \begin{align*}
        \|\sum_{i=1}^n (s_i - \lambda_i) u_i u_i^*\| = \lambda_{\max}(f_{s_1\cdots s_n}) \leq \lambda_{\max} (f_{\emptyset}) \leq 4\sigma .
    \end{align*}
\end{lemma}

\begin{lemma}[Computing the top $k$ coefficients of interlacing polynomials for Theorem \ref{thm:kls20}]\label{lem:maxcoeff-kls20}
Let $\mathcal{F}$ denote the interlacing family defined in Lemma \ref{lem:interlacing-family-kls20-original}.
Given independent random variable $\xi_1,\dots,\xi_n\in \R$ with finite supports such that we know all moments of each random variable. There exists an algorithm $\textsc{MaxCoeff}_{\mathcal{F}}(k,\ell,s_1,\cdots,s_\ell)$ such that for any $1\leq \ell \leq n$ and $1\leq k \leq m$, and for any  $s_1,\dots,s_\ell$ in the supports of $\xi_1,\dots,\xi_\ell$ respectively, returns the top-$k$ coefficients of $f_{s_1,\dots,s_\ell}$ in time $O(n^{k}\poly(m))$.
\end{lemma}
\begin{proof}
The leading constant $\prod_{i=1}^\ell p_{i,s_i}$ can be easily computed. It is easy to see that all odd-degree terms vanish. Thus, for the following expected polynomial:
\begin{align*}
    \E_{\xi_{\ell+1},\dots,\xi_n} \left[ \det\left(x^2 I - \left(\sum_{i=1}^\ell s_i u_iu_i^* + \sum_{i=\ell+1}^n \xi_i u_iu_i^*\right)^2\right) \right],
\end{align*}
consider the coefficient of $x^{2(n-k)}$, which is $(-1)^k$ times the sum of all principal $k$-by-$k$ minors of the matrix $(\sum_{i=1}^\ell s_i u_iu_i^* + \sum_{i=\ell+1}^n \xi_i u_iu_i^*)^2$ in expectation. For any fixed value $\xi_{\ell+1},\dots,\xi_n$, we have
\begin{align*}
    &~\sigma_k\left(\left(\sum_{i=1}^\ell s_i u_iu_i^* + \sum_{i=\ell+1}^n \xi_i u_iu_i^*\right)^2\right)\\
    = &~ \sigma_k\left(\sum_{i,j\in [\ell]} s_i s_j \langle u_i, u_j\rangle u_iu_i^* + \sum_{i\in[\ell],j\in [n]\backslash [\ell]} s_i \xi_j \langle u_i, u_j\rangle (u_i u_j^* + u_ju_i^*)  + \sum_{i,j\in [n]\backslash[\ell]} \xi_i\xi_j \langle u_i, u_j\rangle u_iu_i^*\right).
\end{align*}
For simplicity, let
\begin{align*}
    c_{i,j}:=\begin{cases}
    s_is_j\langle u_i, u_j\rangle&\text{if}~i,j\in [\ell],\\
    \xi_is_j\langle u_i, u_j\rangle&\text{if}~i\in [n]\backslash [\ell],j\in [\ell],\\
    s_i\xi_j\langle u_i, u_j\rangle&\text{if}~i\in [\ell], j\in [n]\backslash [\ell],\\
    \xi_i\xi_j\langle u_i, u_j\rangle&\text{if}~i, j\in [n]\backslash [\ell]
    \end{cases}~~~,~\forall i,j\in [n]\times [n].
\end{align*}
Then, we have
\begin{align*}
    \sigma_k\left(\sum_{i,j\in [n]\times [n]} c_{i,j}u_iu_j^*\right) = &~ \sum_{S\in \binom{[n]\times [n]}{k}}\sigma_k\left(\sum_{(i,j)\in S} c_{i,j} u_iu_j^*\right)\\
    = &~ \sum_{S\in \binom{[n]\times [n]}{k}} \prod_{(i,j)\in S}c_{i,j}\cdot \sigma_k\left(\sum_{(i,j)\in S} u_iu_j^*\right),
\end{align*}
where the first step follows from Proposition 3.11 in \cite{mss15b} and the second step follows from Proposition 3.10 in \cite{mss15b}. 

Thus, the coefficient of the expected polynomial is 
\begin{align*}
    \sum_{S\in \binom{[n]\times [n]}{k}} \E_{\xi_{\ell+1},\dots,\xi_n}\left[\prod_{(i,j)\in S}c_{i,j}\right]\cdot \sigma_k\left(\sum_{(i,j)\in S} u_iu_j^*\right).
\end{align*}
Note that there are at most $n^{2k}$ terms in the summation. And for each $S$, the expectation $\E_{\xi_{\ell+1},\dots,\xi_n}\left[\prod_{(i,j)\in S}c_{i,j}\right]$ can be computed in $O(k)$-time, assuming we know all moments of each random variable, and the inner product $\langle u_i, u_j\rangle$ for all $i,j\in [n]$ can be pre-processed in $O(n^2m)$-time. The sum of minors $\sigma_k(\sum_{(i,j)\in S}u_iu_j^*)$ can be computed in $\poly(m)$-time. 

Therefore, the coefficient of $x^{2(m-k)}$ can be computed in $n^{2k}\cdot \poly(m)$-time, which implies that the top-$k$ coefficients can be computed in $O(n^k\poly(m))$-time. The Lemma is then proved.
\end{proof}

A similar proof of Corollary \ref{cor:subexpalg-kls20-ours} yields the following corollary:

\begin{corollary}[Sub-exponential algorithm for Theorem \ref{thm:kls20}]\label{cor:subexpalg-kls}
    Let $u_1, \dots, u_n \in \Gamma_{+}^h$ be $n$ vectors such that $\sigma^2 = \Big\| \sum_{i=1}^n (u_i u_i^*)^2\|$.
    Let $\mathcal{F}$ be the interlacing family defined in Lemma \ref{lem:interlacing-family-kls20-original}.
    Then for any $\delta > 0$, the algorithm $\mathrm{KadisonSinger}(\delta, \mathcal{F})$ returns a sign assignment $(s_1,\cdots,s_n)\in \{\pm 1\}^n$, such that
    \begin{align*}
        \Big\| \sum_{i=1}^n s_i u_i \Big\|_h \leq 4(1+\delta) \sigma
    \end{align*} 
    in time 
    \begin{align*}
         \left(O(n^{O(\sqrt{m}\log n/\delta)} \poly(m)) + \frac{m \log^2 n}{\delta^2} \right)
         \cdot 2^{\sqrt{m}} \cdot \sqrt{m}.
    \end{align*}
\end{corollary}

\subsection{Sub-exponential algorithm for Theorem \ref{thm:ag14_ours}}\label{sec:alg-ag14-ours}

In this section, we want to describe the sub-exponential algorithm for constructing Theorem \ref{thm:ag14_ours}.
Let $\mathcal{Q}$ denote the interlacing family 
\ifdefined\isarxiv 
defined in Definition \ref{def:interlacing-family-ag14}. 
\else
as follows:
 Let $\mu : 2^{[n]} \to \R$ be a homogeneous strongly Rayleigh probability distribution. Let $v_1,\dots, v_n\in \Gamma_+^h$ be $n$ vectors such that $\rank_h(v_i) \leq 1$ for all $i\in [n]$. Let $\mathcal{F} = \{S \subseteq [n]: \mu(S) > 0\}$ be the support of $\mu$.
    For any $S\in \mathcal{F}$, let 
    \begin{align*}
        q_S(x) = \mu(S) \cdot h\left(xe - \sum_{i \in S} v_i\right).
    \end{align*}
    Let $\mathcal{Q }$ denote the following family of polynomials:
    \begin{align}
        \mathcal{Q} := \Big\{ q_{s_1\cdots s_\ell}(s) = \sum_{\substack{t_{\ell+1} ,\cdots,t_{n} \\ (s_1\cdots s_\ell,t_{\ell+1} ,\cdots,t_{n}) \in \mathcal{F}}} q_{s_1,\cdots,s_\ell, t_{\ell+1},\cdots,t_n} : \forall \ell \in [n], (s_1,\cdots,s_\ell) \in \mathcal{F}|_{[\ell]} \Big\}.
        \label{eq:ag_interlacing}
    \end{align}
    where $\mathcal{F}|_{[\ell]}$ is $\mathcal{F}$ restricted to $[\ell]$, and a subset is represented by a binary indicator vector. %
\fi

Suppose each $q_\mathbf{s}(x) \in \mathcal{Q}$ has degree $d$.
Let $\textsc{MaxCoeff}_\mathcal{Q}(k,\ell,s_1,\cdots,s_\ell)$ be the oracle defined in Definition \ref{def:oracle-max-coeff}, i.e. given any $k\in [d]$, $\ell \in [n]$, $(s_1,\cdots,s_\ell)$ in the support of $\mu$, outputs the top-$k$ coefficients of $q_{s_1,\cdots,s_\ell}$
in at most $\T_{\mathrm{Coeff}}(\mathcal{Q}, k)$ time. The following corollary states that if $\T_{\mathrm{Coeff}}(\mathcal{Q}, k)$ is polynomial in $k$, then we can construct Theorem \ref{thm:ag14_ours}. The proof of this corollary is similar to that of Corollary \ref{cor:subexpalg-kls20-ours}.

\begin{corollary}[Sub-exponential algorithm for Theorem \ref{thm:ag14_ours}, formal statement of Proposition \ref{prop:subexpalg-ag14-informal}]\label{cor:subexpalg-ag14}
    Let $h\in \R[x_1,\dots,x_m]$ denote a hyperbolic polynomial with respect to $e\in \Gamma_{++}^h$. 
    Let $\mu$ be a homogeneous strongly Rayleigh probability distribution on $[n]$ such that the marginal probability of each element is at most $\epsilon_1$, and let $v_1, \cdots, v_n \in \Gamma_{+}^h$ be $n$ vectors such that $\sum_{i=1}^n v_i  = e$, and for all $i \in [n]$, $\| v_i \|_h \leq \epsilon_2$ and $\rank_h(v_i) \leq 1$.

    Let $\mathcal{Q}$ be the interlacing family defined in
    \ifdefined\isarxiv
    Definition \ref{def:interlacing-family-ag14}
    \else 
    \eqref{eq:ag_interlacing}
    \fi
    , and $\textsc{MaxCoeff}_\mathcal{Q}$ be the oracle defined in Definition \ref{def:oracle-max-coeff} with running time $\T_{\mathrm{Coeff}}(\mathcal{Q}, k)$.
    Then for any $\delta > 0$, the algorithm $\mathrm{KadisonSinger}(\delta, \mathcal{Q})$ returns a set $S$ in the support of $\mu$, such that
    \begin{align*}
     \Big\| \sum_{i\in S} u_i \Big\|_h \leq (1+\delta) \cdot \left( 4 (\epsilon_1+ \epsilon_2) + 2 (\epsilon_1+\epsilon_2)^2 \right)
    \end{align*} 
    in time 
    \begin{align*}
         \left(\T_{\mathrm{Coeff}}\left(\mathcal{Q}, O(\log(n) m \delta^{-1})\right) + k^2 \right)
         \cdot 2^{\sqrt{m}} \cdot \sqrt{m}.
    \end{align*}
\end{corollary}

\section{Examples and Discussions}\label{sec:example}

\paragraph*{Examples of real-stable polynomials}
\begin{itemize}
    \item \textit{Spanning tree polynomial:} let $G=(V,E)$ be a connected undirected graph. Then its spanning tree polynomial 
    \begin{align}\label{eq:spanning-tree-polyn}
        P_G(x) = \sum_{\substack{T\subset E,\\T~\text{spanning tree}}} \prod_{e\in T} x_e
    \end{align}
    is real-stable.
    \item \textit{Elementary Symmetric Polynomials: } For any $n,k>0$, the elementary symmetric polynomial
    \begin{align*}
        e_k(x) = \sum_{S\in \binom{[n]}{k}} \prod_{i\in S}x_i 
    \end{align*}
    is real-stable.
    \item \textit{Vertex matching polynomial: } let $G=(V,E)$ be a undirected graph. Then its vertex matching polynomial
    \begin{align}\label{eq:matching-polyn}
        M_G(x)=\sum_{\substack{M\subset E,\\ M ~\text{matching}}} \prod_{\{u,v\}\in M} -x_ux_v
    \end{align}
    is real-stable \cite{bb09}.
    \item \textit{V\'{a}mos matroid polynomial: } let ${\cal B}$ consists of all subsets $B\subset [10]$ of size $4$ except for $\{1, 2, 3, 4\}, \{1, 2, 5, 6\}, \{1, 2, 7, 8\}, \{1, 2, 9, 10\}, \{3, 4, 5, 6\}, \{5, 6, 7, 8\},\{7, 8, 9, 10\}$. Then, ${\cal B}$ is the collection of basis of a V\'{a}mos matroid. See Figure \ref{fig:vamos-matroid-polyn} as an illustration of ${\cal B}$. Its generating polynomial
    \begin{align*}
        f_{10}(x) = \sum_{B\in {\cal B}} \prod_{i\in B}x_i
    \end{align*}
    is real-stable. Furthermore, for any $k>0$, $(f_{10}(x))^k$ cannot be represented as a determinant of a linear matrix with positive semidefinite Hermitian forms \cite{bvy14}.
    \begin{figure}[h]
    \centering 
        \includegraphics[width=0.35\textwidth]{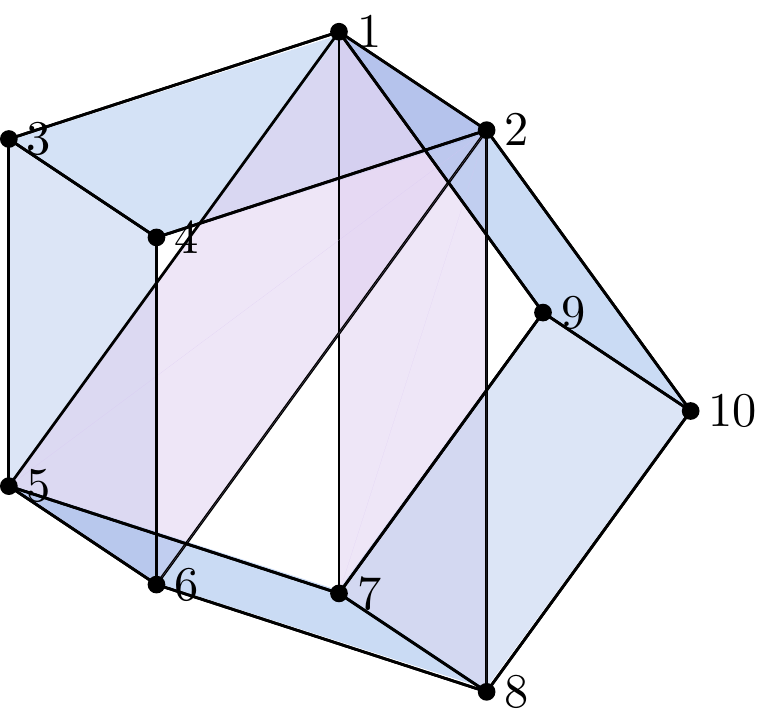}
        \caption{An example of the V\'{a}mos matroid $\mathcal{B}$. $\mathcal{B}$ contains all the sets in $\binom{[10]}{4}$, except the sets of size $4$ represented by the colored faces.}
    \label{fig:vamos-matroid-polyn}
    \end{figure}
    \item \textit{Determinant of the mixture of PSD matrices: } let $A_1,\dots,A_n\in \R^{d\times d}$ be PSD matrix and $B\in \R^{d\times d}$ be symmetric. Then, the polynomial
    \begin{align*}
        p(x)=\det(x_1A_1+\cdots + x_n A_n + B)
    \end{align*}
    is real-stable.
\end{itemize}

Figure \ref{fig:spanning-tree-polyn} and Figure \ref{fig:matching-polyn}  illustrate an example of the spanning tree polynomial (Eqn. (\ref{eq:spanning-tree-polyn})) and vertex matching polynomial (Eqn. (\ref{eq:matching-polyn})) respectively.  

\begin{figure}[h]
\centering 
    \includegraphics[width=0.6\textwidth]{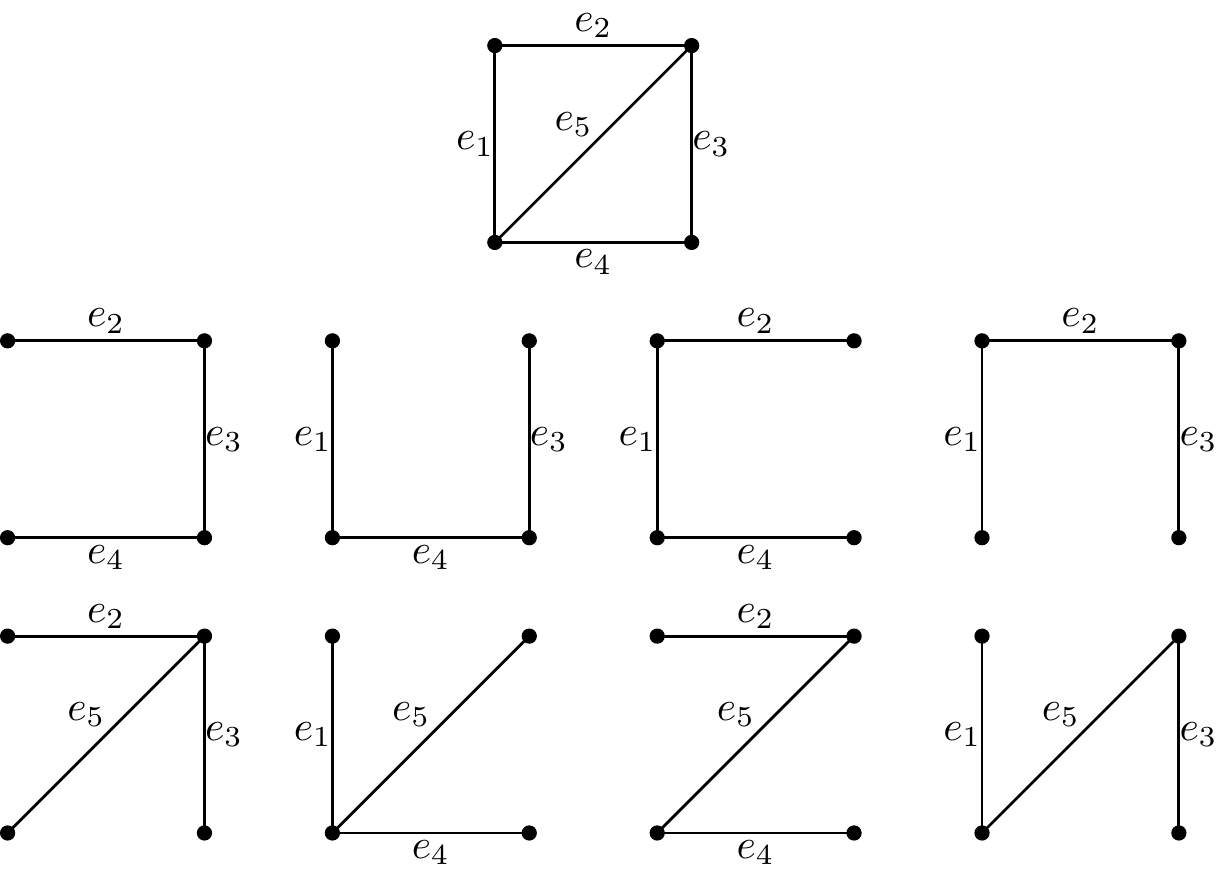}
    \caption{An example of the spanning tree polynomial. Above: the graph $G$; Below: the set of spanning trees of $G$. Therefore the spanning tree polynomial of $G$ is 
    $ P_G(x) = x_2x_3x_4 + x_1x_3x_4 + x_1x_2x_4 + x_1x_2x_3 + x_2x_3x_5 + x_1x_4x_5 + x_2x_4x_5 + x_1x_3x_5. $}
\label{fig:spanning-tree-polyn}
\end{figure}

\begin{figure}[h]
\centering 
    \includegraphics[width=0.45\textwidth]{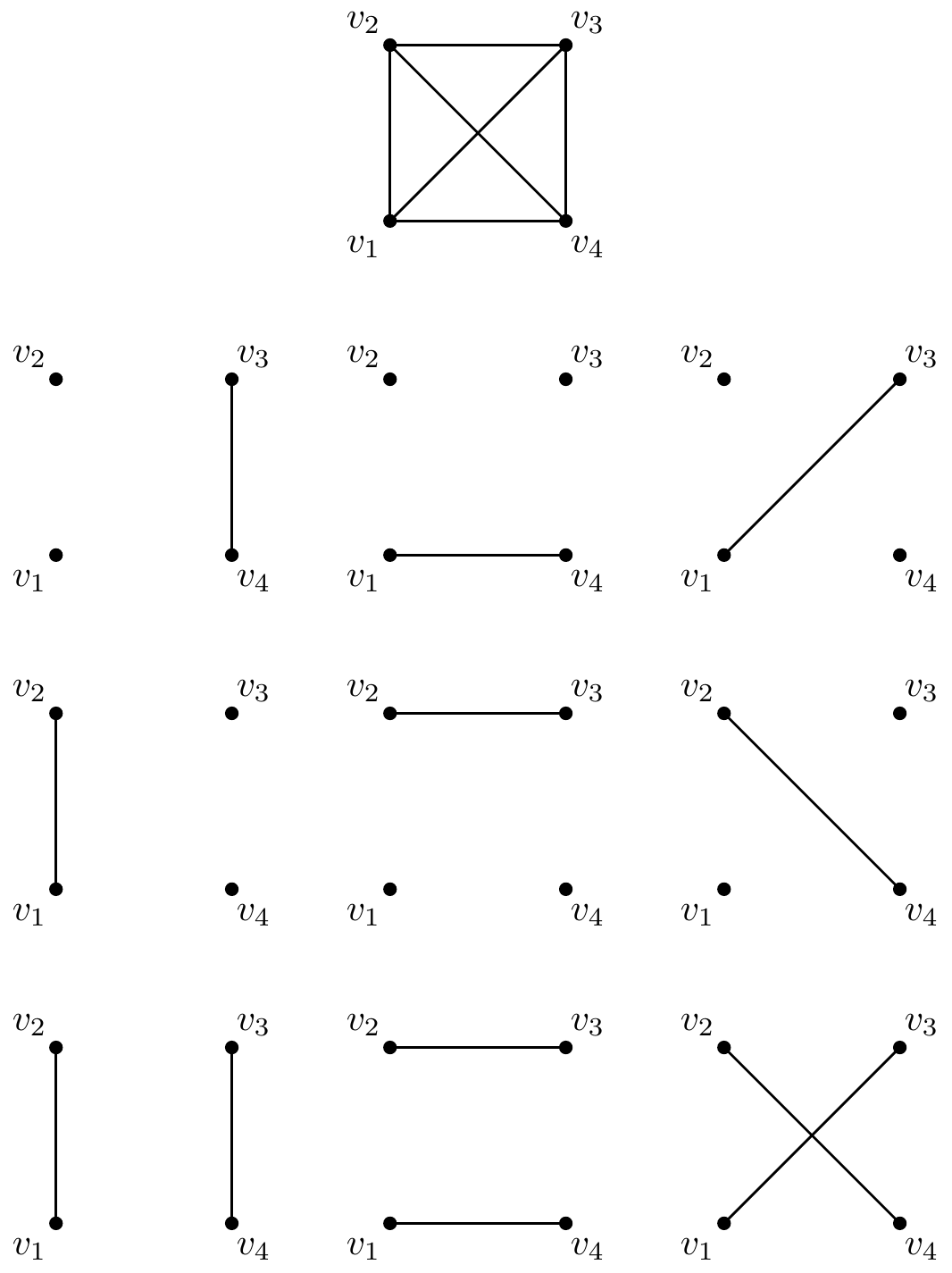}
    \caption{An example of the vertex matching polynomial. Above: the graph $G$; Below: the set of matchings of $G$. Therefore the matching polynomial of $G$ is 
    $ M_G(x) = -(x_1x_2 + x_3x_4 + x_2x_3 + x_1 x_4 + x_1 x_3 + x_2 x_4) + 3\cdot x_1x_2x_3x_4$.}
\label{fig:matching-polyn}
\end{figure}

\paragraph*{Examples of hyperbolic polynomials}
\begin{itemize}
    \item \textit{Lorentz polynomial: }
    \begin{align*}
        p(x)=x_n^2-x_1^2-\cdots - x_{n-1}^2
    \end{align*}
    is hyperbolic with respect to $e=\begin{bmatrix}0 & \cdots & 0 & 1\end{bmatrix}^\top$.
    \item \textit{Determinant polynomial: }
    \begin{align*}
        p(x)=\det(\mathrm{mat}(x))\in \R[\{x_{i,j}\}_{1\leq i\leq j\leq n}]
    \end{align*}
    is hyperbolic with respect to $e=\mathrm{vec}(I)$, where $\mathrm{mat}(\cdot)$ packs an $(n(n+1)/2)$-dimensional vector to an $n$-by-$n$ symmetric matrix, and $\mathrm{vec}(\cdot)$ vectorize a symmetric matrix to a vector.
    \item \textit{Multivariate matching polynomial: } let $G=(V,E)$ be an undirected graph. Then
    \begin{align*}
        \mu_G(x,w)=\sum_{\substack{M\subset E,\\ M ~\text{matching}}}(-1)^{|M|} \cdot \prod_{u\notin V(M)} x_u\cdot \prod_{e\in M} w_e^2
    \end{align*}
    is hyperbolic \cite{ami19} with respect to $e=\begin{bmatrix}{\bf 1}_V & {\bf 0}\end{bmatrix}^\top$.
\end{itemize}
\fi

\end{document}